\documentclass[a4paper,12pt]{article}
\usepackage{amsmath}
\usepackage{amssymb}
\usepackage{amsthm}
\usepackage[english]{babel}
\usepackage{hyperref}
\usepackage[all]{xy}
\usepackage{tikz}
\usetikzlibrary{decorations.pathreplacing}
\usepackage{mathdots}

\newtheorem{theor}{Theorem}[section]
\newtheorem{lemma}[theor]{Lemma}

\newcommand{\soc}{\mathrm{soc}}

\newcommand{\Hom}{\mathrm{Hom}}
\newcommand{\End}{\mathrm{End}}

\newcommand{\id}{\mathrm{id}}

\newcommand{\s}{\Sigma}

\newcommand{\md}{\mbox{-}\mathrm{mod}}
\newcommand{\Md}{\!\mod}

\newcommand{\Z}{\mathbb{Z}}

\renewcommand{\1}{\mathbf{1}}
\newcommand{\sgn}{\mathbf{\mathrm{sgn}}}

\newcommand{\Mull}{{\tt M}}

\newcommand{\la}{\lambda}

\newcommand{\da}{{\downarrow}}
\newcommand{\ua}{{\uparrow}}

\renewcommand{\Im}{\mathrm{Im}}
\renewcommand{\epsilon}{\varepsilon}
\renewcommand{\phi}{\varphi}
\newcommand{\xymat}{\xymatrix@R=6pt@C=10pt}

\begin{document}

\begin{center}
{\Large 
Irreducible tensor products for alternating groups in characteristic 5}

\vspace{12pt}

Lucia Morotti






\end{center}

\begin{abstract}
In this paper we study irreducible tensor products of representations of alternating groups and classify such products in characteristic 5.
\end{abstract}



\section{Introduction}

Let $D_1$ and $D_2$ be irreducible representations of a group $G$. In general the tensor product $D_1\!\otimes\! D_2$ is not irreducible. We say that $D_1\!\otimes\! D_2$ is a non-trivial irreducible tensor product if $D_1\!\otimes\! D_2$ is irreducible and neither $D_1$ nor $D_2$ has dimension 1. The classification of non-trivial irreducible tensor products is relevant to the description of maximal subgroups in finite groups of Lie type, see \cite{a} and \cite{as}.

Non-trivial irreducible tensor product of representations of symmetric groups have been fully classified (see \cite{bk}, \cite{gk}, \cite{gj}, \cite{m1} and \cite{z1}). In particular non-trivial irreducible tensor products for $\s_n$ only exist if $p=2$ and $n\equiv 2\Md 4$. For alternating groups, non-trivial irreducible tensor products have been classified in characteristic 0 in \cite{bk3} and in characteristic $p\geq 7$ in \cite{bk2}. For covering groups of symmetric and alternating groups a partial classification of non-trivial irreducible tensor products can be found in \cite{b2}, \cite{bk4} and \cite{kt}. When considering groups of Lie type in defining characteristic, non-trivial irreducible tensor products are not unusual, due to Steinberg tensor product theorem. In non-defining characteristic however it has been proved that in almost all cases no non-trivial irreducible tensor products exist, see \cite{kt2} and \cite{mt}.

In this paper we will consider the case where $G=A_n$ is an alternating groups. Also we will mostly consider the case $p=5$ in this paper, although some results hold in general, provided $p\not=2$. Our main result is the following:

\begin{theor}\label{t3}
Let $p=5$ and $D_1$ and $D_2$ be irreducible representations of $A_n$ of dimension greater than 1. Then $D_1\!\otimes\! D_2$ is irreducible if and only if $n\not\equiv 0\Md 5$ and, up to exchange, $D_1\cong E^\lambda_\pm$ with $\lambda=\lambda^\Mull$ a JS-partition and $D_2\cong E^{(n-1,1)}$. In this case $E^\lambda_\pm\otimes E^{(n-1,1)}\cong E^\nu$, where $\nu$ is obtained from $\lambda$ by removing the top removable node and adding the bottom addable node.
\end{theor}

In view of Lemma \ref{l23}, this theorem extends the main result of \cite{bk2} to characteristic 5. In smaller characteristic irreducible tensor products of the form $E^\lambda_\pm\otimes E^\mu_\pm$ also exists. For example $E^{(3,2)}_+\otimes E^{(3,2)}_-\cong E^{(4,1)}$ if $p=2$ and $E^{(4,1^2)}_+\otimes E^{(4,1^2)}_-\cong E^{(4,2)}$ if $p=3$ (see \cite{bk2}). A classifications of irreducible tensor products in characteristic 3 and a partial classification in characteristic 2 can be found in \cite{m2}.

To prove the theorem we need to consider three cases:
\begin{enumerate}
\item
$D_1=E^\lambda$ and $D_2=E^\mu$: in this case $D_1\otimes D_2$ is not irreducible by  \cite{bk}.

\item
$D_1=E^\lambda_\pm$ and $D_2=E^\mu$: the proof of this case is covered by Theorems \ref{tsns} and \ref{tsnat}.

\item
$D_1=E^\lambda_\pm$ and $D_2=E^\mu_\pm$: in this case $D_1\otimes D_2$ is not irreducible by Theorem \ref{tds}.
\end{enumerate}

\section{Notations and basic results}

Let $F$ be an algebraically closed field  of characteristic $p$.

For a partition $\lambda\vdash n$ let $S^\lambda$ be the corresponding Specht module, $M^\lambda:=\1\ua_{\s_\alpha}^{\s_n}$ to be the permutation module induced from the Young subgroup $\s_\alpha=\s_{\alpha_1}\times \s_{\alpha_2}\times \ldots\subseteq \s_n$ and let $Y^\lambda$ to be the corresponding Young module. (Notice that $M^\lambda$ can be defined also for unordered partitions). If $\lambda$ is a $p$-regular partition (that is a partition where no part is repeated $p$ or more times) we define $D^\lambda$ to be the irreducible $F\s_n$-module indexed by $\lambda$. The modules $D^\lambda$, $M^\lambda$ and $Y^\lambda$ are known to be self-dual and this fact will be used throughout the paper. Further, from their definition we have that $D^{(n)}\cong S^{(n)}\cong M^{(n)}\cong \1_{\s_n}$. For more informations on such modules see \cite{JamesBook}, \cite{JamesArcata} and Section 4.6 of \cite{Martin}.

Let $\sgn$ be the sign representation of $\s_n$. Since $\sgn$ is $1$-dimensional, if $\la$ is $p$-regular then $D^\lambda\otimes \sgn$ is irreducible. In particular there exists a $p$-regular partition $\lambda^\Mull$ (the Mullineux dual of $\lambda$) such that $D^{\lambda^\Mull}\cong D^\lambda\otimes \sgn$. 
For $p\not=2$ it is well known that if $\lambda\not=\lambda^\Mull$ then $D^\lambda\da_{A_n}=E^\lambda$ is irreducible (and in this case $E^\lambda\cong E^{\lambda^\Mull}$), while if $\lambda=\lambda^\Mull$ then $D^\lambda\da_{A_n}=E^\lambda_+\oplus E^\lambda_-$ is the direct sum of two non-isomorphic irreducible representations of $A_n$. Further all irreducible representations of $A_n$ are of one of these two forms (see for example \cite{f}).

When considering permutation and Young modules we have the following results, where for any partition $\lambda\vdash n$ we define $A_\lambda:=\s_\lambda\cap A_n$. Lemma \ref{lM} holds easily by Mackey's theorem. For a proof of Lemma \ref{LYoung} see \cite{JamesArcata} and Section 4.6 of \cite{Martin}.

\begin{lemma}\label{lM}
If $\lambda\vdash n$ with $\lambda\not=(1^n)$, then $M^\lambda\da_{A_n}\cong 1\ua_{A_\lambda}^{A_n}$.
\end{lemma}

\begin{lemma} \label{LYoung}
There exist indecomposable $F \s_n$-modules $\{Y^\lambda\mid \lambda\vdash n\}$ such that $M^\lambda\cong Y^\lambda\,\oplus\, \bigoplus_{\mu\rhd\lambda}(Y^\mu)^{\oplus m_{\mu,\lambda}}$ for some $m_{\mu,\lambda}\in\Z_{\geq 0}$. Moreover, $Y^\lambda$ can be characterized as the unique direct summand of $M^\lambda$ such that  $S^\lambda\subseteq Y^\lambda$. Finally, we have $(Y^\lambda)^*\cong Y^\lambda$ for all $\lambda\vdash n$.
\end{lemma}

We will now state some known results about branching in characteristic $p$. Let $M$ be a $F\s_n$-module lying in a unique block $B$ with content $(b_0,\ldots,b_{p-1})$ (see \cite{KBook}). For $0\leq i\leq p-1$, define $e_iM$ as the restriction of $M\da_{\s_{n-1}}$ to the block with content $(b_0,\ldots,b_{i-1},b_i-1,b_{i+1},\ldots,b_{p-1})$. Similarly, for $0\leq i\leq p-1$, define $f_iM$ as the restriction of $M\ua^{\s_{n+1}}$ to the block with content $(b_0,\ldots,b_{i-1},b_i+1,b_{i+1},\ldots,b_{p-1})$. Extend then the definition of $e_iM$ and $f_iM$ to arbitrary $F\s_n$-modules additively. The following result holds for example by Theorems 11.2.7 and 11.2.8 of \cite{KBook}.

\begin{lemma}\label{l45}
For $M$ a $F\s_n$-module we have that
\[M\da_{\s_{n-1}}\cong e_0M\oplus\ldots\oplus e_{p-1}M\hspace{24pt}\mbox{and}\hspace{24pt}M\ua^{\s_{n+1}}\cong f_0M\oplus\ldots\oplus f_{p-1}M.\]
\end{lemma}

For $r\geq 1$ let $e_i^{(r)}:F\s_n\md\rightarrow F\s_{n-r}\md$ and $f_i^{(r)}:F\s_n\md\rightarrow F\s_{n+r}\md$ denote the divided power functors (see Section 11.2 of \cite{KBook} for the definitions). For $r=0$ define $e_i^{(0)}D^\lambda$ and $f_i^{(0)}D^\lambda$ to be equal to $D^\lambda$. The modules $e_i^rD^\lambda$ and $e_i^{(r)}D^\lambda$ (and similarly $f_i^rD^\lambda$ and $f_i^{(r)}D^\lambda$) are closely connected as we can be seen in the next two lemmas. For a partition $\lambda$ and $0\leq i\leq 1$ let $\epsilon_i(\lambda)$ be the number of normal nodes of $\lambda$ of residue $i$ and $\phi_i(\lambda)$ be the number of conormal nodes of $\lambda$ of residue $i$ (see Section 11.1 of \cite{KBook} or Section 2 of \cite{bk2} for two different but equivalent definitions of normal and conormal nodes). Normal and conormal nodes of partitions will play a crucial role throughout the paper. If $\epsilon_i(\lambda)\geq 1$ denote by $\tilde{e}_i(\lambda)$ the partition obtained from $\lambda$ by removing the bottom normal node of residue $i$. Similarly, if $\phi_i(\lambda)\geq 1$ denote by $\tilde{f}_i(\lambda)$ the partition obtained from $\lambda$ by adding the top conormal node of residue $i$.

\begin{lemma}\label{l39}
Let $\lambda\vdash n$ be a $p$-regular partition. Also let $0\leq i\leq p-1$ and $r\geq 0$. Then $e_i^rD^\lambda\cong(e_i^{(r)}D^\lambda)^{\oplus r!}$. Further $e_i^{(r)}D^\lambda\not=0$ if and only if $\epsilon_i(\lambda)\geq r$. In this case
\begin{enumerate}
\item\label{l39a}
$e_i^{(r)}D^\lambda$ is a self-dual indecomposable module with head and socle isomorphic to $D^{\tilde{e}_i^r(\lambda)}$,

\item\label{l39b}
$[e_i^{(r)}D^\lambda:D^{\tilde{e}_i^r(\lambda)}]=\binom{\epsilon_i(\lambda)}{r}=\dim\End_{\s_{n-1}}(e_i^{(r)}D^\lambda)$,

\item\label{l39c}
if $D^\psi$ is a composition factor of $e_i^{(r)}D^\lambda$ then $\epsilon_i(\psi)\leq \epsilon_i(\lambda)-r$, with equality holding if and only if $\psi=\tilde{e}_i^r(\lambda)$.
\end{enumerate}
\end{lemma}

\begin{lemma}\label{l40}
Let $\lambda\vdash n$ be a $p$-regular partition. Also let $0\leq i\leq p-1$ and $r\geq 0$. Then $f_i^rD^\lambda\cong(f_i^{(r)}D^\lambda)^{\oplus r!}$. Further $f_i^{(r)}D^\lambda\not=0$ if and only if $\phi_i(\lambda)\geq r$. In this case
\begin{enumerate}
\item\label{l40a}
$f_i^{(r)}D^\lambda$ is a self-dual indecomposable module with head and socle isomorphic to $D^{\tilde{f}_i^r(\lambda)}$,

\item\label{l40b}
$[f_i^{(r)}D^\lambda:D^{\tilde{f}_i^r(\lambda)}]=\binom{\phi_i(\lambda)}{r}=\dim\End_{\s_{n+1}}(f_i^{(r)}D^\lambda)$,

\item\label{l40c}
if $D^\psi$ is a composition factor of $f_i^{(r)}D^\lambda$ then $\phi_i(\psi)\leq \phi_i(\lambda)-r$, with equality holding if and only if $\psi=\tilde{f}_i^r(\lambda)$.
\end{enumerate}
\end{lemma}

For proofs see Theorems 11.2.10 and 11.2.11 of \cite{KBook} (the case $r=0$ holds trivially).
In particular, for $r=1$, we have that $e_i=e_i^{(1)}$ and $f_i=f_i^{(1)}$. In this case there are other compositions factors of $e_iD^\lambda$ and $f_iD^\lambda$ which are known (see Remark 11.2.9 of \cite{KBook}).

\begin{lemma}\label{l56}
Let $\lambda$ be a $p$-regular partition. If $A$ is a normal node of $\lambda$ of residue $i$ and $\lambda\setminus A$ is $p$-regular then $[e_iD^\lambda:D^{\lambda\setminus A}]$ is equal to the number of normal nodes of $\lambda$ of residue $i$ weakly above $A$.

Similarly if $B$ is a conormal node of $\lambda$ of residue $i$ and $\lambda\cup B$ is $p$-regular then $[f_iD^\lambda:D^{\lambda\cup B}]$ is equal to the number of conormal nodes of $\lambda$ of residue $i$ weakly below $B$.
\end{lemma}

The following properties of $e_i$ and $f_i$ are just a special cases of Lemma 8.2.2(ii) and Theorem 8.3.2(i) of \cite{KBook}.

\begin{lemma}\label{l57}
If $M$ is self dual then so are $e_iM$ and $f_iM$.
\end{lemma}

\begin{lemma}\label{l48}
The functors $e_i$ and $f_i$ are left and right adjoint of each others.
\end{lemma}

The first part of the next lemma follows from Lemma 5.2.3 of \cite{KBook}. The second part follows by the definition of $\tilde{e}^r_i$ and $\tilde{f}^r_i$ and from Lemmas \hyperref[l39c]{\ref*{l39}\ref*{l39c}} and \hyperref[l40c]{\ref*{l40}\ref*{l40c}}.

\begin{lemma}\label{l47}
For $r\geq 0$ and $p$-regular partitions $\lambda,\nu$ we have that $\tilde{e}^r_i(\lambda)=D^\nu$ if and only if $\tilde{f}^r_i(\nu)=\lambda$. Further in this case $\epsilon_i(\nu)=\epsilon_i(\lambda)-r$ and $\phi_i(\nu)=\phi_i(\lambda)+r$.
\end{lemma}

When considering the number of normal and conormal nodes of a partition we have the following result (see Lemma 2.8 of \cite{m1}, for $p$-regular partitions it also follows from Lemmas \ref{l45}, \ref{l39}, \ref{l40} and Corollary 4.2 of \cite{k3}):

\begin{lemma}\label{l52}
Any partition has 1 more conormal node than it has normal nodes.
\end{lemma}

Since the modules $e_iD^\lambda$ (or the modules $f_iD^\lambda$) lie in pairwise distinct blocks we have the following result by Lemmas \ref{l45}, \ref{l39} and \ref{l40}.

\begin{lemma}\label{l53}
For a $p$-regular partition $\lambda\vdash n$ we have that
\[\dim\End_{\s_{n-1}}(D^\lambda\da_{\s_{n-1}})=\epsilon_0(\lambda)+\ldots+\epsilon_{p-1}(\lambda)\]
and
\[\dim\End_{\s_{n+1}}(D^\lambda\ua^{\s_{n+1}})=\phi_0(\lambda)+\ldots+\phi_{p-1}(\lambda).\]
\end{lemma}

Partitions with only one normal node will play a special role in this paper. In view of the previous lemma and since irreducible modules of symmetric groups are self-dual, a $p$-regular partition $\la\vdash n$ has only one normal node if and only if $D^\la\da_{\s_{n-1}}$ is irreducible. Such partitions are also called JS-partitions and 
they can be classified as follows (see Section 4 of \cite{JS} and Theorem D of \cite{k2}):

\begin{lemma}
Let $\lambda=(a_1^{b_1},\ldots,a_h^{b_h})$ with $a_1>a_2>\ldots>a_h\geq 1$ and $1\leq b_i\leq p-1$ for $1\leq i\leq h$. Then $\lambda$ is a JS-partition if and only if $a_i-a_{i+1}+b_i+b_{i+1}\equiv 0\Md p$ for each $1\leq i<h$.
\end{lemma}

We will often consider filtrations for certain modules. For arbitrary modules $M_1,\ldots,M_h$ we will write $M\sim M_1|\ldots|M_h$ if $M$ has a filtration with factors $M_1,\ldots,M_h$ counted from the bottom. For irreducible modules $D_1,\ldots,D_h$ we will write $M=D_1|\dots|D_h$ if $M$ is a uniserial module with composition factors $D_1,\ldots,D_h$ counted from the bottom.

Certain results presented in this paper hold only when $\la$ and $\la^\Mull$ have enough parts. For any partition $\lambda$ we will write $h(\lambda)$ for the height of $\la$, that is the number of its parts.

\section{Permutation modules}

In this section we study the structure of certain small permutation modules $M^\alpha$.

\begin{lemma}\label{l1}
Let $1\leq k<p$ and $2k\leq n$. Then
\[M^{(n-k,k)}\sim S^{(n-k,k)}|M^{(n-k+1,k-1)}.\]
\end{lemma}

\begin{proof}
See for example Lemmas 3.1 and 3.2 of \cite{bk5}.
\end{proof}

\begin{lemma}\label{l25}
Let $p=5$ and $n\equiv 1\Md 5$ with $n\geq 6$. Then
\begin{align*}
Y^{(n)}&=D^{(n)}=S^{(n)},\\
Y^{(n-1,1)}&=D^{(n-1,1)}=S^{(n-1,1)},\\
Y^{(n-2,2)}&=\overbrace{D^{(n)}|D^{(n-2,2)}}^{S^{(n-2,2)}}|\overbrace{D^{(n)}}^{S^{(n)}},\\
Y^{(n-3,3)}&=D^{(n-3,3)}=S^{(n-3,3)},\\
Y^{(n-2,1^2)}&=D^{(n-2,1^2)}=S^{(n-2,1^2)},\\
Y^{(n-3,2,1)}&\sim\overbrace{D^{(n-2,2)}|D^{(n-3,2,1)}}^{S^{(n-3,2,1)}}|\overbrace{D^{(n)}|D^{(n-2,2)}}^{S^{(n-2,2)}},\\
Y^{(n-3,1^3)}&=D^{(n-3,1^3)}=S^{(n-3,1^3)},
\end{align*}
Further
\begin{align*}
M^{(n)}\cong& Y^{(n)},\\
M^{(n-1,1)}\cong& Y^{(n-1,1)}\oplus Y^{(n)},\\
M^{(n-2,2)}\cong& Y^{(n-2,2)}\oplus Y^{(n-1,1)},\\
M^{(n-3,3)}\cong& Y^{(n-3,3)}\oplus Y^{(n-2,2)}\oplus Y^{(n-1,1)},\\
M^{(n-2,1^2)}\cong& Y^{(n-2,1^2)}\oplus Y^{(n-2,2)}\oplus (Y^{(n-1,1)})^2,\\
M^{(n-3,2,1)}\cong& Y^{(n-3,2,1)}\oplus Y^{(n-2,1^2)}\oplus Y^{(n-3,3)}\oplus Y^{(n-2,2)}\oplus (Y^{(n-1,1)})^2,\\
M^{(n-3,1^3)}\cong& Y^{(n-3,1^3)}\oplus (Y^{(n-3,2,1)})^2\oplus (Y^{(n-2,1^2)})^3\oplus Y^{(n-3,3)}\\
&\oplus Y^{(n-2,2)}\oplus (Y^{(n-1,1)})^3.
\end{align*}
\end{lemma}

\begin{proof}
Notice first that all the considered simple modules lie in pairwise distinct blocks, apart for $D^{(n)}$, $D^{(n-2,2)}$ and $D^{(n-3,2,1)}$ all three of which lie in the same block. From Theorem 24.15 of \cite{JamesBook} and from \cite{jw} we have that $[S^{(n-2,2)}:D^{(n)}]=1$, $[S^{(n-3,2,1)}:D^{(n)}]=0$ and $[S^{(n-3,2,1)}:D^{(n-2,2)}]=1$. It follows that the structure of the Specht modules is as given in the lemma. Further, since the Young modules are indecomposable and self-dual it is easy to see that the Young modules structure is also as given in the lemma, apart possibly for the structure of $Y^{(n-3,2,1)}$.

From block decomposition we have that $D^{(n-3,2)}\cong S^{(n-3,2)}$ is a direct summand of $M^{(n-3,2)}$. In particular $D^{(n-3,2)}\ua^{\s_n}$ is a direct summand of $M^{(n-3,2,1)}$. Notice that since $n\equiv 1\Md 5$,
\[\begin{tikzpicture}
\draw (0,0) node {0};
\draw (0,-0.4) node {4};
\draw (0,-0.8) node {3};
\draw (0.4,0) node {1};
\draw (0.4,-0.4) node {0};
\draw (0.8,0) node {2};
\draw (0.8,-0.4) node {1};
\draw (1.6,0) node {2};
\draw (2,0) node {3};

\draw (-1.5,-0.2) node {$(n-3,2)=$};

\draw (2.3,-0.4) node {.};

\draw (-0.2,0.2) -- (1.8,0.2) -- (1.8,-0.2) -- (0.6,-0.2) -- (0.6,-0.6) -- (-0.2,-0.6) -- (-0.2,0.2);
\end{tikzpicture}\]
So, from Lemmas \ref{l45} and \ref{l40}, from Corollary 17.14 of \cite{JamesBook} and from block decomposition we have that
\begin{align*}
D^{(n-3,2)}\ua^{\s_n}&\cong S^{(n-3,2)}\ua^{\s_n}\\
&\sim S^{(n-3,2,1)}|S^{(n-3,3)}|S^{(n-2,2)}\\
&\sim (\overbrace{S^{(n-3,2,1)}|S^{(n-2,2)}}^{f_3D^{(n-3,2)}})\oplus \overbrace{S^{(n-3,3)}}^{f_1D^{(n-3,2)}}.
\end{align*}
Since $f_3D^{(n-3,2)}$ is indecomposible by Lemma \ref{l40}, it follows that $f_3D^{(n-3,2)}\cong Y^{(n-3,2,1)}$ by Lemma \ref{LYoung}.

The multiplicities of the Young modules as direct summands of the modules $M^\alpha$ follow by comparing multiplicities of composition factors and from 14.1 of \cite{JamesBook} or by Section 3 of \cite{h} if $\alpha$ is a $2$-parts partition.
\end{proof}

\begin{lemma}\label{l26}
Let $p=5$ and $n\equiv 4\Md 5$ with $n\geq 9$. Then
\begin{align*}
Y^{(n)}&=D^{(n)}=S^{(n)},\\
Y^{(n-1,1)}&=D^{(n-1,1)}=S^{(n-1,1)},\\
Y^{(n-2,2)}&=D^{(n-2,2)}=S^{(n-2,2)},\\
Y^{(n-3,3)}&=\overbrace{D^{(n-2,2)}|D^{(n-3,3)}}^{S^{(n-3,3)}}|\overbrace{D^{(n-2,2)}}^{S^{(n-2,2)}},\\
Y^{(n-2,1^2)}&=D^{(n-2,1^2)}=S^{(n-2,1^2)},\\
Y^{(n-3,2,1)}&=D^{(n-3,2,1)}=S^{(n-3,2,1)},\\
Y^{(n-3,1^3)}&=D^{(n-3,1^3)}=S^{(n-3,1^3)}.
\end{align*}
Further
\begin{align*}
M^{(n)}\cong& Y^{(n)},\\
M^{(n-1,1)}\cong& Y^{(n-1,1)}\oplus Y^{(n)},\\
M^{(n-2,2)}\cong& Y^{(n-2,2)}\oplus Y^{(n-1,1)}\oplus Y^{(n)},\\
M^{(n-3,3)}\cong& Y^{(n-3,3)}\oplus Y^{(n-1,1)}\oplus Y^{(n)},\\
M^{(n-2,1^2)}\cong& Y^{(n-2,1^2)}\oplus Y^{(n-2,2)}\oplus (Y^{(n-1,1)})^2\oplus Y^{(n)},\\
M^{(n-3,2,1)}\cong& Y^{(n-3,2,1)}\oplus Y^{(n-2,1^2)}\oplus Y^{(n-3,3)}\oplus Y^{(n-2,2)}\oplus (Y^{(n-1,1)})^2\oplus Y^{(n)},\\
M^{(n-3,1^3)}\cong& Y^{(n-3,1^3)}\oplus (Y^{(n-3,2,1)})^2\oplus (Y^{(n-2,1^2)})^3\oplus Y^{(n-3,3)}\oplus (Y^{(n-2,2)})^2\\
&\oplus (Y^{(n-1,1)})^3\oplus Y^{(n)}.
\end{align*}
\end{lemma}

\begin{proof}
Notice first that all the considered simple modules lie in pairwise distinct blocks, apart for $D^{(n-2,2)}$ and $D^{(n-3,3)}$ which lie in the same block. From Theorem 24.15 of \cite{JamesBook} we have that $[S^{(n-3,3)}:D^{(n-2,2)}]=1$. The structures of the Specht modules then follow. Further, since the Young modules are indecomposable and self-dual it is easy to see that the Young modules structure is also as given in the lemma. The multiplicities of the Young modules as direct summands of the modules $M^\alpha$ follow by comparing multiplicities of composition factors and from 14.1 of \cite{JamesBook} or by Section 3 of \cite{h} if $\alpha$ is a $2$-parts partition.
\end{proof}

\begin{lemma}\label{l24}
Let $p=5$ and $n\equiv 0\Md 5$ with $n\geq 10$. Then
\begin{align*}
Y^{(n)}&=D^{(n)}=S^{(n)},\\
Y^{(n-1,1)}&=\overbrace{D^{(n)}|D^{(n-1,1)}}^{S^{(n-1,1)}}|\overbrace{D^{(n)}}^{S^{(n)}},\\
Y^{(n-2,2)}&=D^{(n-2,2)}=S^{(n-2,2)},\\
Y^{(n-3,3)}&=D^{(n-3,3)}=S^{(n-3,3)},\\
Y^{(n-4,4)}&=\overbrace{D^{(n-2,2)}|D^{(n-4,4)}}^{S^{(n-2,2)}}|\overbrace{D^{(n-2,2)}}^{S^{(n-2,2)}},\\
Y^{(n-2,1^2)}&\sim\overbrace{D^{(n-1,1)}|D^{(n-2,1^2)}}^{S^{(n-2,1^2)}}|\overbrace{D^{(n)}|D^{(n-1,1)}}^{S^{(n-1,1)}},\\
Y^{(n-3,2,1)}&=D^{(n-3,2,1)}=S^{(n-3,2,1)},\\
Y^{(n-4,3,1)}&=\overbrace{D^{(n-3,2,1)}|D^{(n-4,3,1)}}^{S^{(n-4,3,1)}}|\overbrace{D^{(n-3,2,1)}}^{S^{(n-3,2,1)}},\\
Y^{(n-4,2^2)}&=D^{(n-4,2^2)}=S^{(n-4,2^2)}.
\end{align*}
Further
\begin{align*}
M^{(n)}\cong& Y^{(n)},\\
M^{(n-1,1)}\cong& Y^{(n-1,1)},\\
M^{(n-2,2)}\cong& Y^{(n-2,2)}\oplus Y^{(n-1,1)},\\
M^{(n-3,3)}\cong& Y^{(n-3,3)}\oplus Y^{(n-2,2)}\oplus Y^{(n-1,1)},\\
M^{(n-4,4)}\cong& Y^{(n-4,4)}\oplus Y^{(n-3,3)}\oplus Y^{(n-1,1)},\\
M^{(n-2,1^2)}\cong& Y^{(n-2,1^2)}\oplus Y^{(n-2,2)}\oplus Y^{(n-1,1)},\\
M^{(n-3,2,1)}\cong& Y^{(n-3,2,1)}\oplus Y^{(n-2,1^2)}\oplus Y^{(n-3,3)}\oplus (Y^{(n-2,2)})^2\oplus Y^{(n-1,1)},\\
M^{(n-4,3,1)}\cong& Y^{(n-4,3,1)}\oplus Y^{(n-2,1^2)}\oplus Y^{(n-4,4)}\oplus (Y^{(n-3,3)})^2\oplus Y^{(n-2,2)}\\
&\oplus Y^{(n-1,1)},\\
M^{(n-4,2^2)}\cong& Y^{(n-4,2^2)}\oplus Y^{(n-4,3,1)}\oplus Y^{(n-3,2,1)}\oplus Y^{(n-2,1^2)}\oplus Y^{(n-4,4)}\\
&\oplus (Y^{(n-3,3)})^2\oplus (Y^{(n-2,2)})^2\oplus Y^{(n-1,1)}.
\end{align*}
\end{lemma}

\begin{proof}
We have the following subsets of pairwise distinct blocks: $\{D^{(n)},$ $D^{(n-1,1)},D^{(n-1^2)}\}$, $\{D^{(n-2,2)},D^{(n-4,4)}\}$, $\{D^{(n-3,3)}\}$, $\{D^{(n-3,2,1)},D^{(n-4,3,1)}\}$ and $\{D^{(n-4,2^2)}\}$. The structure of the Specht modules then follows by Theorems 24.1 and 24.15 of \cite{JamesBook} and by \cite{jw}. Further, since the Young modules are indecomposable and self-dual it is easy to see that the Young modules structure is also as given in the lemma, apart possibly for the structure of $Y^{(n-2,1^2)}$. For $Y^{(n-2,1^2)}$ note that $D^{(n-2,1)}\cong S^{(n-2,1)}$ by Lemma \ref{l26}. Since $n\equiv 0\Md 5$,
\[\begin{tikzpicture}
\draw (0,0) node {0};
\draw (0,-0.4) node {4};
\draw (0,-0.8) node {3};
\draw (0.4,0) node {1};
\draw (0.4,-0.4) node {0};
\draw (0.8,0) node {2};
\draw (1.6,0) node {2};
\draw (2,0) node {3};

\draw (-1.5,-0.2) node {$(n-2,1)=$};

\draw (2.3,-0.4) node {.};

\draw (-0.2,0.2) -- (1.8,0.2) -- (1.8,-0.2) -- (0.2,-0.2) -- (0.2,-0.6) -- (-0.2,-0.6) -- (-0.2,0.2);
\end{tikzpicture}\]
So, from Lemmas \ref{l45} and \ref{l40}, from Corollary 17.14 of \cite{JamesBook} and from block decomposition we have that
\begin{align*}
D^{(n-2,1)}\ua^{\s_n}&\cong S^{(n-2,1)}\ua^{\s_n}\\
&\sim S^{(n-2,1^2)}|S^{(n-2,2)}|S^{(n-1,1)}\\
&\sim (\overbrace{S^{(n-2,1^2)}|S^{(n-1,1)}}^{f_3D^{(n-2,1)}})\oplus \overbrace{S^{(n-2,2)}}^{f_0D^{(n-2,1)}}.
\end{align*}
Since $f_3D^{(n-2,1)}$ is indecomposible by Lemma \ref{l40}, it follows that $f_3D^{(n-2,1)}\cong Y^{(n-2,1^2)}$ by Lemma \ref{LYoung}.

The multiplicities of the Young modules as direct summands of the modules $M^\alpha$ follow by comparing multiplicities of composition factors and from 14.1 of \cite{JamesBook} or by Section 3 of \cite{h} if $\alpha$ is a $2$-parts partition.
\end{proof}

\section{Restrictions to $\s_{n-2,2}$}

We will now consider certain submodules of $D^\lambda\da_{\s_{n-2,2}}$. The next lemma generalizes Lemma 1.2 of \cite{bk2} and will used in studying such restrictions.

\begin{lemma}\label{l16}
Let $M_1,\ldots,M_h,X,Y$ be $FG$ modules. Assume that that $M_1\oplus\ldots\oplus M_h\subseteq X\oplus Y$ and that $M_i$ has simple socle for each $1\leq i\leq h$. Then there exist $I_X,I_Y$ disjoint with $I_X\cup I_Y=\{1,\ldots, h\}$ such that, up to isomorphism, $\oplus_{i\in I_X}M_i\subseteq X$ and $\oplus_{i\in I_Y}M_i\subseteq Y$.
\end{lemma}

\begin{proof}
Let $\pi_X$ and $\pi_Y$ be the projections to $X$ and $Y$ respectively. Since $\pi_X+\pi_Y=\id$ and the modules $M_i$ have simple socles, we can find disjoint sets $I_X,I_Y$ with $I_X\cup I_Y=\{1,\ldots,h\}$ such that $\pi_X$ and $\pi_Y$ are injective on $\oplus_{i\in I_X}\soc(M_i)$ and $\oplus_{i\in I_Y}\soc(M_i)$ respectively. It follows that $\pi_1$ and $\pi_2$ are injective on $\oplus_{i\in I_X}M_i$ and $\oplus_{i\in I_Y}M_i$ respectively and so the lemma holds.
\end{proof}

\begin{lemma}\label{l10}
Let $p\geq 3$ and $\lambda\vdash n$ be $p$-regular. For $0\leq i<p$ we have that $e_i^{(2)}(D^\lambda)\otimes  (D^{(2)}\oplus D^{(1^2)})$ is a direct summand of $D^\lambda\da_{\s_{n-2,2}}$.
\end{lemma}

\begin{proof}
From Lemma \ref{l39} we can assume that $\epsilon_i(\lambda)\geq 2$ (else $e_i^{(2)}D^\lambda=0$).

By definition $e_i^2D^\lambda$ is a block component of $D^\lambda\da_{\s_{n-2}}$. Comparing block decomposition of $D^\lambda\da_{\s_{n-2}}$ and $D^\lambda\da_{\s_{n-2,2}}$, there exist a module $M$ which is a direct sum of $D^\lambda\da_{\s_{n-2,2}}$ with $M\da_{\s_{n-2}}\cong e_i^2D^\lambda$. Notice that $M$ is the sum of the block components of $D^\lambda\da_{\s_{n-2,2}}$ lying in the blocks of $D^{\tilde{e}_i^2(\lambda)}\otimes D^{(2)}$ and of $D^{\tilde{e}_i^2(\lambda)}\otimes D^{(1^2)}$. From Lemmas \ref{l39} and Lemma 1.1 of \cite{bk2} we have that
\[\soc(M)\da_{\s_{n-2}}\cong\soc(e_i^2D^\lambda)\cong D^{\tilde{e}_i^2(\lambda)}\oplus D^{\tilde{e}_i^2(\lambda)}.\]
We will first show that $\soc(M)\cong D^{\tilde{e}_i^2(\lambda)}\otimes(D^{(2)}\oplus D^{(1^2)})$. By definition of $M$, in order to do this it is enough to prove that
\[[\soc(D^\lambda\da_{\s_{n-2,2}}):D^{\tilde{e}_i^2(\lambda)}\otimes D^{(2)}]=1.\]
From Lemma \ref{l40}, by definition of $f^{(2)}_i$ and considering block decomposition we have that
\begin{align*}
\dim\Hom_{\s_{n-2,2}}(D^{\tilde{e}_i^2(\lambda)}\!\otimes\! D^{(2)},D^\lambda\da_{\s_{n-2,2}})
&\!=\!\dim\Hom_{\s_n}((D^{\tilde{e}_i^2(\lambda)}\!\otimes\! D^{(2)})\ua^{\s_n},D^\lambda)\\
&\!=\!\dim\Hom_{\s_n}(f_i^{(2)}(D^{\tilde{e}_i^2(\lambda)}),D^\lambda)\\
&\!=\!1.
\end{align*}
So $\soc(M)\cong D^{\tilde{e}_i^2(\lambda)}\otimes(D^{(2)}\oplus D^{(1^2)})$. Since $D^{(2)}$ and $D^{(1^2)}$ lie in distinct blocks of $\s_2$ and since $\s_2$ is semisimple (as $p\geq 3$), we have that $M\cong(M_1\otimes D^{(2)})\oplus (M_2\otimes D^{(1^2)})$ for some modules $M_1,M_2$  with socle $D^{\tilde{e}_i^2(\lambda)}$. In particular
\[M_1\oplus M_2\cong M\da_{\s_{n-2}}\cong e_i^{(2)}D^\lambda\oplus e_i^{(2)}D^\lambda.\]
From Lemma \ref{l16} it follows that $M_1$ and $M_2$ are isomorphically contained in $e_i^{(2)}D^\lambda$ and so, comparing dimensions, that $M_1,M_2\cong e_i^{(2)}D^\lambda$.
\end{proof}

\begin{lemma}\label{l11}
Let $p\geq 3$ and $\lambda\vdash n$ be $p$-regular. For each $j$ with $\epsilon_j(\lambda)>0$ and for each $i\not=j$ there exists $b_{i,j}\in\{D^{(2)},D^{(1^2)}\}$ such that
\[\bigoplus_{{j:\epsilon_j(\lambda)>0}\atop{i\not=j}}e_i(D^{\tilde{e}_j(\lambda)})\otimes b_{i,j}\]
is both a submodule and a quotient of $D^\lambda\da_{\s_{n-2,2}}$.
\end{lemma}

\begin{proof}
Since $D^\lambda\da_{\s_{n-2,2}}$, $e_i(D^{\tilde{e}_j(\lambda)})$, $D^{(2)}$ and $D^{(1^2)}$ are self-dual it is enough to show that there exist $b_{i,j}$ such that
\[\bigoplus_{{j:\epsilon_j(\lambda)>0}\atop{i\not=j}}e_i(D^{\tilde{e}_j(\lambda)})\otimes b_{i,j}\subseteq D^\lambda\da_{\s_{n-2,2}}.\]
Since $p\geq 3$, there exist $M_1,M_2$ with $D^\lambda\da_{\s_{n-2,2}}\cong (M_1\otimes D^{(2)})\oplus (M_2\otimes D^{(1^2)})$. From Lemmas \ref{l45} and \ref{l39}
\[\bigoplus_{{j:\epsilon_j(\lambda)>0}\atop{i\not=j}}e_i(D^{\tilde{e}_j(\lambda)})\subseteq \bigoplus_{i\not=j}e_ie_jD^\lambda\subseteq D^\lambda\da_{\s_{n-2}}\cong M_1\oplus M_2.\]
and the modules $e_i(D^{\tilde{e}_j(\lambda)})$ have simple socle, if they are non-zero. The lemma then follows by Lemma \ref{l16}.
\end{proof}

\section{Dimensions of homomorphism rings}

In this section we study the size of certain homomorphism rings, which will allow us later in Sections \ref{s1} and \ref{s2} to prove that in almost all cases the tensor product of two irreducible representations of $A_n$ is not irreducible.

\begin{lemma}\label{l2}
For any $F\s_n$-module $V$ and any $\alpha\vdash n$ we have that
\[\dim\Hom_{\s_n}(M^\alpha,\End_F(V))=\dim\End_{\s_\alpha}(V\da_{\s_\alpha}).\]
\end{lemma}

\begin{proof}
This follows by Frobenius reciprocity, since $M^\alpha=1\ua_{\s_\alpha}^{\s_n}$.
\end{proof}

\begin{lemma}\label{l6}
Let $p\geq 3$. If $\lambda=\lambda^\Mull$ a $p$-regular partition and $V$ is an $F\s_n$-module, then
\begin{align*}
&\dim\Hom_{A_n}(V\da_{A_n},\Hom_F(E^\lambda_+\oplus E^\lambda_-,E^\lambda_\pm))\\
&=\dim\Hom_{A_n}(\Hom_F(E^\lambda_\pm,E^\lambda_+\oplus E^\lambda_-),V^*\da_{A_n})\\
&=\dim\Hom_{\s_n}(V,\End_F(D^\lambda)).
\end{align*}
\end{lemma}

\begin{proof}
Using Frobenious reciprocity we have
\begin{align*}
\Hom_{A_n}(V\da_{A_n},\Hom_F(E^\lambda_+\oplus E^\lambda_-,E^\lambda_\pm))\hspace{-1pt}&\cong\hspace{-1pt}\Hom_{A_n}(V\da_{A_n},(E^\lambda_+\oplus E^\lambda_-)^*\otimes E^\lambda_\pm)\\
&\cong\hspace{-1pt}\Hom_{\s_n}(V,((E^\lambda_+\oplus E^\lambda_-)^*\otimes E^\lambda_\pm)\ua^{\s_n})\\
&\cong\hspace{-1pt}\Hom_{\s_n}(V,(D^\lambda)^*\otimes D^\lambda)\\
&\cong\hspace{-1pt}\Hom_{\s_n}(V,\End_F(D^\lambda)).
\end{align*}
The other equality holds similarly.
\end{proof}

The next lemma will play a major role in Sections \ref{s1} and \ref{s2} to prove that most tensor products are not irreducible. The order on partitions appearing in the lemma is the lexicographic order.

\begin{lemma}\label{l15}
Let $G=\s_n$ or $G=A_n$ and let $V$ and $W$ be $FG$-modules. For $\alpha\vdash n$ let $m_{\alpha}$ be such that there exist $\phi^\alpha_1,\ldots,\phi^\alpha_{m_\alpha}\in\Hom_G(M^\alpha,V^*)$ with $\phi^\alpha_1|_{S^\alpha},\ldots,\phi^\alpha_{m_\alpha}|_{S^\alpha}$ linearly independent. Similarly let $n_\alpha$ be such that there exist $\psi^\alpha_1,\ldots,\psi^\alpha_{n_\alpha}\in\Hom_G(M^\alpha,W)$ with $\psi^\alpha_1|_{S^\alpha},\ldots,\psi^\alpha_{n_\alpha}|_{S^\alpha}$ linearly independent. If $G=\s_n$ let $A$ be the set of all $p$-regular partitions of $n$. If $G=A_n$ let $A$ be the set of $p$-regular partitions $\alpha\vdash n$ with $\alpha>\alpha^\Mull$. Then
\[\dim\Hom_G(V,W)\geq\sum_{\alpha\in A}m_{\alpha}n_{\alpha}.\]
\end{lemma}

\begin{proof}
If $G=\s_n$ then by Corollary 12.2 of \cite{JamesBook}, the head of $S^\alpha$ is the unique composition factor of $M^\alpha$ which is isomorphic to $D^\alpha$ and all other composition factors of $M^\alpha$ are of the form $D^\beta$ with $\beta>\alpha$. If $G=A_n$ and $\alpha\in A$ it then follows that the head of $S^\alpha$ is the unique composition factor of $M^\alpha$ which is isomorphic to $E^\alpha$ and all other composition factors of $M^\alpha$ are of the form $E^\beta$ or $E^\beta_\pm$ with $\beta>\alpha$.

Let $(\phi^\alpha_i)^*\in\Hom_G(V,M^\alpha)$ be the dual of $\phi^\alpha_i$ for $1\leq i\leq m_\alpha$. Let $B_i:=\ker(\phi^\alpha_i)^*$. Note that for each $i$, $V/B_i$ is a submodule of $M^\alpha$  and there exists $C_i\supseteq B_i$ with $C_i/B_i\cong S^\alpha$. By the previous paragraph and assumption we have that $C_i\subseteq B_j$ for $j\not=i$. It then follows that the functions $\psi^\alpha_j\circ(\phi^\alpha_i)^*$, with $1\leq i\leq m_\alpha$ and $1\leq j\leq n_\alpha$ are linearly independent and if $f$ is a non-zero linear combination of them then $\Im(f)$ has a composition factor $D^\alpha$ or $E^\alpha$ and all other composition factors of $\Im(f)$ are indexed by partitions $\beta>\alpha$.


It then follows that the functions $\psi^\alpha_j\circ(\phi^\alpha_i)^*$, with $\alpha\in A$, $1\leq i\leq m_{\alpha}$, $1\leq j\leq n_{\alpha}$ are linearly independent (using induction on the minimal $\alpha\in A$ which indexes a composition factor of the image of a linear combination of such functions) and so the lemma holds.
\end{proof}

The following lemmas will be used to prove that in certain cases there exists $\phi\in\Hom_{\s_n}(M^\alpha,\End_F(D^\lambda))$ which does not vanish on $S^\alpha$. The existence of such homomorphisms will then be used to apply Lemma \ref{l15}.

\begin{lemma}\label{c2}
Let $p=5$ and $n\equiv \pm 1\Md 5$ with $n\geq 6$. If $\lambda\vdash n$ and
\begin{align*}
\dim\End_{\s_{n-3}}(\hspace{-1pt}D^\lambda\da_{\s_{n-3}}\hspace{-1pt})\hspace{-1pt}\!>&2\dim\End_{\s_{n-3,2}}(\hspace{-1pt}D^\lambda\da_{\s_{n-3,2}}\hspace{-1pt})\!+\!\dim\End_{\s_{n-2}}(\hspace{-1pt}D^\lambda\da_{\s_{n-2}}\hspace{-1pt})\\
&-\!\dim\End_{\s_{n-3,3}}(\hspace{-1pt}D^\lambda\da_{\s_{n-3,3}}\hspace{-1pt})\!-\!\dim\End_{\s_{n-2,2}}(\hspace{-1pt}D^\lambda\da_{\s_{n-2,2}}\hspace{-1pt})\\
&-\!\dim\End_{\s_{n-1}}(\hspace{-1pt}D^\lambda\da_{\s_{n-1}}\hspace{-1pt})\!+\!1,
\end{align*}
then there exists $\psi\in\Hom_{\s_n}(M^{(n-3,1^3)},\End_F(D^\lambda))$ which does not vanish on $S^{(n-3,1^3)}$.
\end{lemma}

\begin{proof}
From Lemmas \ref{l25} and \ref{l26} we have that in either case $M^{(n-3,1^3)}\cong S^{(n-3,1^3)}\oplus A$. Since $Y^{(n-3,1^3)}\cong S^{(n-3,1^3)}$ and comparing multiplicities of Young modules appear as direct summands of the permutation modules we also have that
\[A\oplus M^{(n-3,3)}\oplus M^{(n-2,2)}\oplus M^{(n-1,1)}\cong (M^{(n-3,2)})^2\oplus M^{(n-2,1^2)}\oplus M^{(n)}.\]
The result then follows from Lemma \ref{l2}.
\end{proof}

\begin{lemma}\label{c1'}
Let $p=5$ and $n\equiv 0\Md 5$ with $n\geq 10$. If $\lambda\vdash n$ and
\begin{align*}
\dim\End_{\s_{n-4,2^2}}\hspace{-1pt}(\hspace{-1pt}D^\lambda\da_{\s_{n-4,2^2}}\hspace{-1pt})\!>&\dim\End_{\s_{n-4,3}}\hspace{-1pt}(\hspace{-1pt}D^\lambda\da_{\s_{n-4,3}}\hspace{-1pt})\!\hspace{-0.5pt}+\!\dim\End_{\s_{n-3,2}}\hspace{-1pt}(\hspace{-1pt}D^\lambda\da_{\s_{n-3,2}}\hspace{-1pt})\!\hspace{-0.5pt}\\
&+\!\dim\End_{\s_{n-2,2}}\hspace{-1pt}(\hspace{-1pt}D^\lambda\da_{\s_{n-2,2}}\hspace{-1pt})\!\hspace{-0.5pt}-\!\dim\End_{\s_{n-3,3}}\hspace{-1pt}(\hspace{-1pt}D^\lambda\da_{\s_{n-3,3}}\hspace{-1pt})\!\hspace{-0.5pt}\\
&-\!\dim\End_{\s_{n-2}}\hspace{-1pt}(\hspace{-1pt}D^\lambda\da_{\s_{n-2}}\hspace{-1pt})
\end{align*}
then there exists $\psi\in\Hom_{\s_n}(M^{(n-4,2^2)},\End_F(D^\lambda))$ which does not vanish on $S^{(n-4,2^2)}$.
\end{lemma}

\begin{proof}
From Lemma \ref{l24} we have that in either case $M^{(n-4,2^2)}\cong S^{(n-4,2^2)}\oplus A$ since $Y^{(n-4,2^2)}\cong S^{(n-4,2^2)}$. Comparing multiplicities of Young modules appear as direct summands of the permutation modules we also have that
\[A\oplus M^{(n-3,3)}\oplus M^{(n-2,1^2)}\cong M^{(n-4,3,1)}\oplus M^{(n-3,2,1)}\oplus M^{(n-2,2)}.\]
The result then follows from Lemma \ref{l2}.
\end{proof}

Lemmas \ref{c2} and \ref{c1'} will be checked to hold for some particular classes of partitions in Sections \ref{s3} and \ref{s4}.

\begin{lemma}\label{l1l2}
Let $1\leq k<p$, $n\geq 2k$ and $\lambda\vdash n$ be $p$-regular. If
\[x:=\dim\End_{\s_{n-k,k}}(D^\lambda\da_{\s_{n-k,k}})-\dim\End_{\s_{n-k+1,k-1}}(D^\lambda\da_{\s_{n-k+1,k-1}}),\]
then there exist $\psi_j:M^{(n-k,k)}\rightarrow\End_F(D^\lambda)$ for $1\leq j\leq x$ such that $\psi_1|_{S^{(n-k,k)}},\ldots,\psi_x|_{S^{(n-k,k)}}$ are linearly independent.
\end{lemma}

\begin{proof}
It follows from Lemmas \ref{l1} and \ref{l2}.
\end{proof}

In the remaining part of this section we will study the (in)equalities appearing in the previous lemmas and prove that they hold in for certain families of partitions.

\begin{lemma}\label{l30}
Let $p\geq 3$, $n\geq 4$ and $\lambda\vdash n$ be $p$-regular with $\lambda\not=(n),(n)^\Mull$. Then
\[\dim\End_{\s_{n-2,2}}(D^\lambda\da_{\s_{n-2,2}})>\dim\End_{\s_{n-1}}(D^\lambda\da_{\s_{n-1}}).\]
\end{lemma}

\begin{proof}
See Theorem 3.3 of \cite{ks}.
\end{proof}

We will now prove that, in most cases, the inequality in the previous lemma can be improved.

\begin{lemma}\label{l12}
Let $\alpha$ and $\beta$ be partitions such that $\alpha$ is obtained from $\beta$ by removing a $j$-node. If $i\not=j$ then all normal $i$-nodes of $\beta$ are also normal in $\alpha$ and all conormal $i$-nodes of $\alpha$ are also conormal in $\beta$.
\end{lemma}

\begin{proof}
As $i\not=j$ all removable $i$-nodes of $\beta$ are also removable in $\alpha$ and all addable $i$-nodes of $\alpha$ are also addable in $\beta$. The lemma then follows from the definition of normal and conormal nodes.
\end{proof}

\begin{lemma}\label{l13}
Let $p\geq 3$ and $\lambda\vdash n$ be $p$-regular. If $\epsilon_j(\lambda)>0$. Then
\begin{align*}
\dim\End_{\s_{n-2,2}}(D^\lambda\da_{\s_{n-2,2}})&\geq\sum_i\epsilon_i(\lambda)(\epsilon_i(\lambda)-1)+\sum_{{j:\epsilon_j(\lambda)>0}\atop{i\not=j}}\epsilon_i(\tilde{e}_j(\lambda))\\
&\geq\sum_i\epsilon_i(\lambda)(\epsilon_i(\lambda)-2+|\{j:\epsilon_j(\lambda)>0\}|).
\end{align*}
\end{lemma}

\begin{proof}
From Lemma \ref{l45} we have that
\[D^\lambda\da_{\s_{n-2}}=\bigoplus_{i,j}e_je_i(D^\lambda)=\bigoplus_ie_i^2(D^\lambda)\oplus\bigoplus_{i\not=j}e_ie_j(D^\lambda).\]
From block decomposition and from Lemmas \ref{l10} and \ref{l11} we have that, for certain $b_{i,j}\in\{D^{(2)},D^{(1^2)}\}$
\[B:=\bigoplus_i(e_i^{(2)}(D^\la)\otimes (D^{(2)}\oplus D^{(1^2)}))\oplus \bigoplus_{{j:\epsilon_j(\lambda)>0}\atop{i\not=j}}(e_i(D^{\tilde{e}_j(\lambda)})\otimes b_{i,j})\]
is both a submodule and a quotient of $D^\lambda\da_{\s_{n-2,2}}$. In particular, from Lemma \ref{l39},
\begin{align*}
\dim\End_{\s_{n-2,2}}(D^\lambda\da_{\s_{n-2,2}})&\geq\dim\End_{\s_{n-2,2}}(B)\\
&\geq\sum_i \dim\End_{\s_{n-2,2}}(e_i^{(2)}(D^\la)\otimes (D^{(2)}\oplus D^{(1^2)}))\\
&+ \sum_{{j:\epsilon_j(\lambda)>0}\atop{i\not=j}}\dim\End_{\s_{n-2,2}}(e_i(D^\lambda_j)\otimes b_{i,j})\\
&=\sum_i\epsilon_i(\lambda)(\epsilon_i(\lambda)-1)+\sum_{{j:\epsilon_j(\lambda)>0}\atop{i\not=j}}\epsilon_i(\tilde{e}_j(\lambda)).
\end{align*}
From Lemma \ref{l12} we also have that if $\epsilon_j(\lambda)>0$ then $\epsilon_i(\tilde{e}_j(\lambda))\geq\epsilon_i(\lambda)$ for $i\not=j$. So
\begin{align*}
&\sum_i\epsilon_i(\lambda)(\epsilon_i(\lambda)-1)+\sum_{{j:\epsilon_j(\lambda)>0}\atop{i\not=j}}\epsilon_i(\tilde{e}_j(\lambda))\\
&\geq\sum_i\epsilon_i(\lambda)(\epsilon_i(\lambda)-1)+\sum_{{j:\epsilon_j(\lambda)>0}\atop{i\not=j}}\epsilon_i(\lambda)\\
&=\sum_i\epsilon_i(\lambda)(\epsilon_i(\lambda)-2)+\sum_{j:\epsilon_j(\lambda)>0}\sum_i\epsilon_i(\lambda)\\
&=\sum_i\epsilon_i(\lambda)(\epsilon_i(\lambda)-2+|\{j:\epsilon_j(\lambda)>0\}|).
\end{align*}
\end{proof}

The next lemma compares normal and conormal nodes of $\la$ and $\la^\Mull$ and will be used to apply Lemma \ref{l13} in certain situations. A proof of it could also be obtained using Theorems 4.2 and 4.7 of \cite{k4}.

\begin{lemma}\label{l17}
For any partition $\lambda$ and for any residue $i$,
\[\epsilon_i(\lambda)=\epsilon_{-i}(\lambda^\Mull)\hspace{12pt}\mbox{and}\hspace{12pt}\phi_i(\lambda)=\phi_{-i}(\lambda^\Mull).\]
If $\epsilon_i(\lambda)>0$ then $\tilde{e}_i(\lambda)^\Mull=\tilde{e}_{-i}(\lambda^\Mull)$, while if $\phi_i(\lambda)>0$ then $\tilde{f}_i(\lambda^\Mull)=\tilde{f}_{-i}(\lambda^\Mull)$.
\end{lemma}

\begin{proof}
This follows from Lemma \ref{l39} and by comparing block decomposition of $D^\lambda\da_{\s_{n-1}}$ and of $D^{\lambda^\Mull}\da_{\s_{n-1}}\cong D^\lambda\da_{\s_{n-1}}\otimes\sgn$ (or of $D^\lambda\ua^{\s_{n+1}}$ and of $D^{\lambda^\Mull}\ua^{\s_{n+1}}\cong D^\lambda\ua^{\s_{n+1}}\otimes\sgn$).
\end{proof}

\begin{lemma}\label{l20}
Let $p\geq 3$ and $\lambda$ be $p$-regular. If $\lambda$ has at least 3 normal nodes then
\[\dim\End_{\s_{n-2,2}}(D^\mu\da_{\s_{n-2,2}})>\dim\End_{\s_{n-1}}(D^\mu\da_{\s_{n-1}})+1.\]
If further $\lambda=\lambda^\Mull$ then
\[\dim\End_{\s_{n-2,2}}(D^\mu\da_{\s_{n-2,2}})>\dim\End_{\s_{n-1}}(D^\mu\da_{\s_{n-1}})+2.\]
\end{lemma}

\begin{proof}
From Lemmas \ref{l53} and \ref{l13} it is enough to prove that
\[\sum_i\epsilon_i(\lambda)(\epsilon_i(\lambda)-3+|\{j:\epsilon_j(\lambda)>0\}|)>1\]
or
\[\sum_i\epsilon_i(\lambda)(\epsilon_i(\lambda)-3+|\{j:\epsilon_j(\lambda)>0\}|)>2\]
when $\lambda$ has at least 3 normal nodes (the last inequality only when $\lambda=\lambda^\Mull$).

Assume first that $|\{j:\epsilon_j(\lambda)>0\}|=1$ and let $k$ with $\epsilon_k(\lambda)>0$. Then $\epsilon_i(\lambda)\geq 3$ and so
\[\sum_i\epsilon_i(\lambda)(\epsilon_i(\lambda)-3+|\{j:\epsilon_j(\lambda)>0\}|)=\epsilon_k(\lambda)(\epsilon_k(\lambda)-2)\geq \epsilon_k(\lambda)\geq 3.\]

Assume next that $|\{j:\epsilon_j(\lambda)>0\}|=2$ and let $k\not=l$ with $\epsilon_k(\lambda),\epsilon_l(\lambda)>0$. We can assume that $\epsilon_k(\lambda)\geq 2$. Then
\begin{align*}
\sum_i\epsilon_i(\lambda)(\epsilon_i(\lambda)-3+|\{j:\epsilon_j(\lambda)>0\}|)&=\epsilon_k(\lambda)(\epsilon_k(\lambda)-1)+\epsilon_l(\lambda)(\epsilon_l(\lambda)-1)\\
&\geq\epsilon_k(\lambda)\\
&\geq 2.
\end{align*}
Assume now that $\lambda=\lambda^\Mull$. Then from Lemma \ref{l17}, we have that $k=-l$ and $\epsilon_k(\lambda)=\epsilon_l(\lambda)\geq 2$. In this case
\begin{align*}
\sum_i\epsilon_i(\lambda)(\epsilon_i(\lambda)-3+|\{j:\epsilon_j(\lambda)>0\}|)&=\epsilon_k(\lambda)(\epsilon_k(\lambda)-1)+\epsilon_l(\lambda)(\epsilon_l(\lambda)-1)\\
&\geq2\epsilon_k(\lambda)\\
&\geq 4.
\end{align*}

Assume last that $|\{j:\epsilon_j(\lambda)>0\}|\geq 3$ and let $k,l,h$ pairwise different with $\epsilon_k(\lambda),\epsilon_l(\lambda),\epsilon_h(\lambda)>0$. Then
\[\sum_i\epsilon_i(\lambda)(\epsilon_i(\lambda)-3+|\{j:\epsilon_j(\lambda)>0\}|)\geq\epsilon_k(\lambda)^2+\epsilon_l(\lambda)^2+\epsilon_h(\lambda)^2\geq 3.\]
\end{proof}

We will now prove that Lemma \ref{l30} can be extended in many cases also when $\la$ has only two normal nodes, in particular for Mullineux-fixed partitions with two normal nodes. This will though require a more careful analysis than the at least three normal nodes case and will also require some properties of Mullineux-fixed partitions with two normal nodes.

\begin{lemma}\label{l18}
Let $p\geq 3$, $n\geq 4$ and $\lambda=\lambda^\Mull\vdash n$ be a partition with exactly 2 normal nodes. If there exist $i\not=j$ with $\epsilon_i(\lambda),\epsilon_j(\lambda)\not=0$ then $\tilde{e}_i(\lambda)$ and $\tilde{e}_j(\lambda)$ are not JS-partitions.
\end{lemma}

\begin{proof}
Assume that $\tilde{e}_i(\lambda)$ and $\tilde{e}_j(\lambda)$ are JS-partitions. Then, from Lemmas \ref{l45} and \ref{l39}, $D^\lambda\da_{\s_{n-2}}$ has only two composition factors. Since $\lambda=\lambda^\Mull$ it follows that
\[D^\lambda\da_{\s_{n-2,2}}\cong (D^\nu\otimes D^{(2)})\oplus (D^{\nu^\Mull}\otimes D^{(1^2)})\]
for a certain partition $\nu$. Due to Lemma \ref{l53} this contradicts Lemma \ref{l30}.

Since $\lambda=\lambda^\Mull$ we have from Lemma \ref{l17} that $i=-j$ and that $\tilde{e}_i(\lambda)$ and $\tilde{e}_j(\lambda)$ have the same number of normal nodes. In particular neither $\tilde{e}_i(\lambda)$ nor $\tilde{e}_j(\lambda)$ is a JS-partition.
\end{proof}

\begin{lemma}\label{l21}
Let $p\geq 3$ and $\lambda$ be a $p$-regular partition with 2 normal nodes. Assume that there exist $i\not=j$ with $\epsilon_i(\lambda),\epsilon_j(\lambda)=1$. If
\[\dim\End_{\s_{n-2,2}}(D^\lambda\da_{\s_{n-2,2}})=\dim\End_{\s_{n-1}}(D^\lambda\da_{\s_{n-1}})+1,\]
then, up to exchange, $\tilde{e}_i(\lambda)$ is a JS-partition and $\tilde{e}_j(\lambda)$ has at most 2 normal nodes. Also $\lambda\not=\lambda^\Mull$.
\end{lemma}

\begin{proof}
From Lemma \ref{l39} we have that $\epsilon_i(\tilde{e}_i(\lambda))=\epsilon_i(\lambda)-1=0$ and similarly $\epsilon_j(\tilde{e}_j(\lambda))=0$. So from Lemmas \ref{l53} and \ref{l13} and by assumption
\begin{align*}
\sum_k\epsilon_k(\tilde{e}_i(\lambda))+\sum_k\epsilon_k(\tilde{e}_j(\lambda))&=\sum_{k\not=i}\epsilon_k(\tilde{e}_i(\lambda))+\sum_{k\not=j}\epsilon_k(\tilde{e}_j(\lambda))\\
&\leq \dim\End_{\s_{n-2,2}}(D^\mu\da_{\s_{n-2,2}})\\
&=\dim\End_{\s_{n-1}}(D^\mu\da_{\s_{n-1}})+1\\
&=3.
\end{align*}
So $\tilde{e}_i(\lambda)$ and $\tilde{e}_j(\lambda)$ have in total at most 3 normal nodes, from which the first part of the lemma follows. The second part follows from Lemma \ref{l18}.
\end{proof}

\begin{lemma}\label{l22}
Let $p\geq 3$ and $\lambda$ be a $p$-regular partition with 2 normal nodes. Assume that there exists $i$ with $\epsilon_i(\lambda)=2$ and let $\nu$ be obtained from $\lambda$ by removing the top removable node of $\lambda$. If
\[\dim\End_{\s_{n-2,2}}(D^\lambda\da_{\s_{n-2,2}})=\dim\End_{\s_{n-1}}(D^\lambda\da_{\s_{n-1}})+1,\]
then $\tilde{e}_i(\lambda)$ is a JS-partition and $\nu$ is either a JS-partition or it is not $p$-regular. Also $\lambda\not=\lambda^\Mull$.
\end{lemma}

\begin{proof}
Notice first that from Lemma \ref{l10}
\[D^\lambda\da_{\s_{n-2,2}}=(e_i^{(2)}(D^\lambda)\otimes (D^{(2)}\oplus D^{(1^2)}))\oplus M\]
for a certain module $M$. Comparing block decompositions of $D^\lambda\da_{\s_{n-2}}$ and $D^\lambda\da_{\s_{n-2,2}}$ we have that
\[M\da_{\s_{n-2}}\cong\bigoplus_{(j,k)\not=(i,i)}e_je_k(D^\lambda).\]
Also from Lemma \ref{l39}
\[\dim\End_{\s_{n-2,2}}(e_i^{(2)}(D^\lambda)\otimes (D^{(2)}\oplus D^{(1^2)}))=\epsilon_i(\lambda)(\epsilon_i(\lambda)-1)=2.\]
Notice that $M$ is self-dual, since it is the sum of certain block components of $D^\lambda\da_{\s_{n-2,2}}$. So, if $M$ is non-zero and not simple, then $\dim\End_{\s_{n-2,2}}(M)\geq 2$ (simple modules of $\s_{n-2,2}$ are also self-dual) and so from Lemma \ref{l53}
\[\dim\End_{\s_{n-2,2}}(D^\lambda\da_{\s_{n-2,2}})\geq 2+2>3=\dim\End_{\s_{n-1}}(D^\lambda\da_{\s_{n-1}})+1,\]
contradicting the assumptions. As all simple $\s_2$-modules are 1-dimensional, $M$ is non-zero and not simple if and only if $M\da_{\s_{n-2,2}}\cong\oplus_{(j,k)\not=(i,i)}e_je_k(D^\lambda)$ is non-zero and not simple. In order to prove the lemma it is then enough to prove that $\oplus_{(j,k)\not=(i,i)}e_je_k(D^\lambda)$ is non-zero and not simple,  when $\lambda$ is not as in the text of the lemma.

First assume that $\tilde{e}_i(\lambda)$ is not a JS-partition. Then, from Lemma \ref{l47}, there exists $l\not=i$ with $\epsilon_l(\lambda_i)\geq 1$. So, from Lemma \ref{l39},
\begin{align*}
[\bigoplus_{(j,k)\not=(i,i)}e_je_k(D^\lambda):D^{\tilde{e}_l\tilde{e}_i(\lambda)}]&\geq[e_i(D^\lambda):D^{\tilde{e}_i(\lambda)}]\cdot [e_l(D^{\tilde{e}_i(\lambda)}):D^{\tilde{e}_l\tilde{e}_i(\lambda)}]\\
&=\epsilon_i(\lambda)\epsilon_l(\tilde{e}_i(\lambda))\\
&\geq 2.
\end{align*}
In particular $\oplus_{(j,k)\not=(i,i)}e_je_k(D^\lambda)$ is non-zero and not simple.

Assume next that $\nu$ is $p$-regular but not a JS-partition. From Lemmas \ref{l39} and \ref{l56} we have that $D^\nu$ is a composition component of $e_i(D^\lambda)$ and that $\epsilon_i(\nu)\leq\epsilon_i(\lambda)-2=0$. So $\sum_{l\not=i}\epsilon_l(\nu)\geq 2$ and then
\begin{align*}
\sum_{l\not=i}[\bigoplus_{(j,k)\not=(i,i)}e_je_k(D^\lambda):D^{\tilde{e}_l(\nu)}]&\geq\sum_{l\not=i}[e_i(D^\lambda):D^\nu]\cdot [e_l(D^\nu):D^{\tilde{e}_l(\nu)}]\\
&\geq\sum_{l\not=i}\epsilon_l(\nu)\\
&\geq 2.
\end{align*}
So also in this case $\oplus_{(j,k)\not=(i,i)}e_je_k(D^\lambda)$ is non-zero and not simple.

Assume now that $\lambda=\lambda^\Mull$. Notice that $\nu=\lambda\setminus A$, where $A$ is the top removable node of $\lambda$. Assume first that $\nu$ is not $p$-regular. Then $\lambda_1=\lambda_p+1$. This contradicts $\lambda=\lambda^\Mull$, by Lemma 2.2 of \cite{bkz}. So we can assume that $\nu$ is $p$-regular. Further from Lemma \ref{l17} we have that $i=0$, so that $\epsilon_0(\nu)=0$. In particular there exist $l\not=0$ such that $e_l(D^\nu)\not=0$. Since $D^\nu$ is a component of $e_0(D^\lambda)$, we then have that $e_le_0(D^\lambda)\not=0$. Since $\lambda=\lambda^\Mull$ we also have that $e_{-l}e_0(D^\lambda)\not=0$. As $l\not=0$ and so $l\not=-l$ as $p\geq 3$ is odd, it follows that $\oplus_{(j,k)\not=(i,i)}e_je_k(D^\lambda)$ is non-zero and not simple.
\end{proof}

\section{Special homomorphisms}

In Lemma \ref{c1'} we found a condition under which there exists an  homomorphism $M^{(n-4,2^2)}\to\End_F(D^\la)$ which does not vanish on $S^{(n-4,2^2)}$. The condition presented there is though not always easy to check. Further the lemma only holds for $n\equiv 0\Md 5$. In this section we will present a different method for proving the existence of such homomorphisms and show that the corresponding condition is satisfied by a large class of partitions.



\begin{lemma}\label{ll12}
Let $p\geq 3$, $n\geq 6$ and $V$ be a $F\s_n$-module. If
\[x_{2^2}=\sum_{g\in \s_{3,3}}\sgn(g)g(2,5)(3,6)g^{-1}\]
and $x_{2^2}V\not=0$ then there exists $\psi:M^{(n-4,2^2)}\to\End_F(V)$ which does not vanish on $S^{(n-4,2^2)}$.
\end{lemma}

\begin{proof}
Let $\{v_{\{x,y\},\{z,w\}}|x,y,z,w\in\{1,\ldots,n\}\mbox{ distinct}\}$ be the standard basis of $M^{(n-4,2^2)}$. Define $\psi:M^{(n-4,2^2)}\to\End_F(V)$ through
\[\psi(v_{\{x,y\},\{z,w\}})(a)=(x,y)(z,w)a\]
for each $a\in V$. Let $e$ be the basis element of $S^{(n-4,2^2)}$ corresponding to the tableau
\[\begin{array}{ccccc}
1&4&7&\ldots&n\\
2&5&&&\\
3&6&&&
\end{array}\]
(see \cite[Section 8]{JamesBook} for definition of $e$). Then $\psi(e)(a)=
x_{2^2}a$, from which the lemma follows.
\end{proof}

\begin{lemma}\label{l51}
Let $p=5$, $n\geq 6$ and $\la\vdash n$ be 5-regular with $h(\la),h(\la^\Mull)\geq 3$. Then there exists $\psi:M^{(n-4,2^2)}\to\End_F(D^\lambda)$ which does not vanish on $S^{(n-4,2^2)}$.
\end{lemma}

\begin{proof}
From Theorem 2.8 of \cite{bk5} we have that $D^{(4,1^2)}$ or $D^{(3,1^3)}$ is a composition factor of $D^\la\da_{\s_6}$. So it is enough to prove that $x_{2^2}D^{(4,1^2)}$ and $x_{2^2}D^{(3,1^3)}$ are non-zero, where $x_{2^2}$ is as in Lemma \ref{ll12}. Notice that $D^{(4,1^2)}\cong S^{(4,1^2)}$ and $D^{(3,1^3)}\cong S^{(3,1^3)}$. Let $\{v_{a,b}\}$, $\{e_{a,b}\}$, $\{v_{a,b,c}\}$ and $\{e_{a,b,c}\}$ be the standard bases of $M^{(4,1^2)}$, $S^{(4,1^2)}$, $M^{(3,1^3)}$ and $S^{(3,1^3)}$ respectively. It can be checked that $x_{2^2}e_{2,4}$ has non-zero coefficient for $v_{2,5}$ and that $x_{2^2}e_{2,3,4}$ has non-zero coefficient on $v_{1,5,6}$ and so the lemma holds.
\end{proof}

\section{Partitions of the form $(a+b,a)$ with $b$ small}\label{s3}

Partitions of the form $(a+b,a)$ with $0\leq b\leq 3$ will play a special role in the proof of Theorem \ref{t3}, since for these partitions Corollary 4.12 of \cite{bk5} (which will be used later in the proof of Theorem \ref{tsns}) does not apply. So we will now study the simple modules labeled by such partitions and their restrictions to certain submodules of $\s_n$ more in details.

\begin{lemma}\label{l27}
Let $p=5$ and $\lambda=(a+b,a)\vdash n$ with $0\leq b\leq 3$. If $k\leq 4$ and $a\geq k$ then $D^\lambda\da_{\s_{n-k,k}}$ is given by
\begin{align*}
D^{(a,a)}\da_{\s_{2a-1}}&\!\cong\! D^{(a,a-1)},\\
D^{(a,a)}\da_{\s_{2a-2,2}}&\!\cong\! (\hspace{-1pt}D^{(a,a-2)}\!\otimes\! D^{(2)}\hspace{-1pt})\!\oplus\!(\hspace{-1pt}D^{(a-1,a-1)}\!\otimes\! D^{(1^2)}\hspace{-1pt}),\\
D^{(a,a)}\da_{\s_{2a-3,3}}&\!\cong\! (\hspace{-1pt}D^{(a,a-3)}\!\otimes\! D^{(3)}\hspace{-1pt})\!\oplus\!(\hspace{-1pt}D^{(a-1,a-2)}\!\otimes\! D^{(2,1)}\hspace{-1pt}),\\
D^{(a,a)}\da_{\s_{2a-4,4}}&\!\cong\! (\hspace{-1pt}D^{(a-1,a-3)}\!\otimes\! D^{(3,1)}\hspace{-1pt})\!\oplus\!(\hspace{-1pt}D^{(a-2,a-2)}\!\otimes\! D^{(2^2)}\hspace{-1pt}),\\
D^{(a+1,a)}\da_{\s_{2a}}&\!\cong\! D^{(a+1,a-1)}\!\oplus\! D^{(a,a)},\\
D^{(a+1,a)}\da_{\s_{2a-1,2}}&\!\cong\! (\hspace{-1pt}D^{(a+1,a-2)}\!\otimes\! D^{(2)}\hspace{-1pt})\!\oplus\!(\hspace{-1pt}D^{(a,a-1)}\!\otimes\! D^{(2)}\hspace{-1pt})\!\oplus\!(\hspace{-1pt}D^{(a,a-1)}\!\otimes\! D^{(1^2)}\hspace{-1pt}),\\
D^{(a+1,a)}\da_{\s_{2a-2,3}}&\!\cong\! (\hspace{-1pt}D^{(a,a-2)}\!\otimes\! D^{(3)}\hspace{-1pt})\!\oplus\!(\hspace{-1pt}D^{(a,a-2)}\!\otimes\! D^{(2,1)}\hspace{-1pt})\!\oplus\!(\hspace{-1pt}D^{(a-1,a-1)}\!\otimes\! D^{(2,1)}\hspace{-1pt}),\\
D^{(a+1,a)}\da_{\s_{2a-3,4}}&\!\cong\! (\hspace{-1pt}D^{(a,a-3)}\!\otimes\! D^{(3,1)}\hspace{-1pt})\!\oplus\!(\hspace{-1pt}D^{(a-1,a-2)}\!\otimes\! D^{(3,1)}\hspace{-1pt})\!\oplus\!(\hspace{-1pt}D^{(a-1,a-2)}\!\otimes\! D^{(2^2)}\hspace{-1pt}),\\
D^{(a+2,a)}\da_{\s_{2a+1}}&\!\cong\! D^{(a+1,a)}\!\oplus\! D^{(a+2,a-1)},\\
D^{(a+2,a)}\da_{\s_{2a,2}}&\!\cong\! (\hspace{-1pt}D^{(a,a)}\!\otimes\! D^{(2)}\hspace{-1pt})\!\oplus\!(\hspace{-1pt}D^{(a+1,a-1)}\!\otimes\! D^{(2)}\hspace{-1pt})\!\oplus\!(\hspace{-1pt}D^{(a+1,a-1)}\!\otimes\! D^{(1^2)}\hspace{-1pt}),\\
D^{(a+2,a)}\da_{\s_{2a-1,3}}&\!\cong\! (\hspace{-1pt}D^{(a,a-1)}\!\otimes\! D^{(3)}\hspace{-1pt})\!\oplus\!(\hspace{-1pt}D^{(a,a-1)}\!\otimes\! D^{(2,1)}\hspace{-1pt})\!\oplus\!(\hspace{-1pt}D^{(a+1,a-2)}\!\otimes\! D^{(2,1)}\hspace{-1pt}),\\
D^{(a+2,a)}\da_{\s_{2a-2,4}}&\!\cong\! (\hspace{-1pt}D^{(a-1,a-1)}\!\otimes\! D^{(3,1)}\hspace{-1pt})\!\oplus\!(\hspace{-1pt}D^{(a,a-2)}\!\otimes\! D^{(3,1)}\hspace{-1pt})\!\oplus\!(\hspace{-1pt}D^{(a,a-2)}\!\otimes\! D^{(2^2)}\hspace{-1pt}),\\
D^{(a+3,a)}\da_{\s_{2a+2}}&\!\cong\! D^{(a+2,a)},\\
D^{(a+3,a)}\da_{\s_{2a+1,2}}&\!\cong\! (\hspace{-1pt}D^{(a+1,a)}\!\otimes\! D^{(2)}\hspace{-1pt})\!\oplus\!(\hspace{-1pt}D^{(a+2,a-1)}\!\otimes\! D^{(1^2)}\hspace{-1pt}),\\
D^{(a+3,a)}\da_{\s_{2a,3}}&\!\cong\! (\hspace{-1pt}D^{(a,a)}\!\otimes\! D^{(3)}\hspace{-1pt})\!\oplus\!(\hspace{-1pt}D^{(a+1,a-1)}\!\otimes\! D^{(2,1)}\hspace{-1pt}),\\
D^{(a+3,a)}\da_{\s_{2a-1,4}}&\!\cong\! (\hspace{-1pt}D^{(a,a-1)}\!\otimes\! D^{(3,1)}\hspace{-1pt})\!\oplus\!(\hspace{-1pt}D^{(a+1,a-2)}\!\otimes\! D^{(2^2)}\hspace{-1pt}).
\end{align*}
\end{lemma}

\begin{proof}
For $k\leq 3$ see Lemmas 4.1, 4.5 and 4.7 of \cite{bk5}. Further if $a\geq 4$, from the same lemmas,
\begin{align*}
D^{(a,a)}\da_{\s_{2a-4,2^2}}\!\cong &(D^{(a-2,a-2)}\!\otimes\! D^{(2)}\!\otimes\! D^{(2)})\!\oplus\!(D^{(a-1,a-3)}\!\otimes\! D^{(2)}\!\otimes\! D^{(2)})\\
&\oplus\!(D^{(a-1,a-3)}\!\otimes\! D^{(1^2)}\!\otimes\! D^{(2)})\!\oplus\!(D^{(a-1,a-3)}\!\otimes\! D^{(2)}\!\otimes\! D^{(1^2)})\\
&\oplus\!(D^{(a-2,a-2)}\!\otimes\! D^{(1^2)}\!\otimes\! D^{(1^2)}),\\
D^{(a+1,a)}\da_{\s_{2a-3,2^2}}\!\cong & (D^{(a-1,a-2)}\!\otimes\! D^{(2)}\!\otimes\! D^{(2)})^2\!\oplus\!(D^{(a,a-3)}\!\otimes\! D^{(1^2)}\!\otimes\! D^{(2)})\\
&\oplus\!(D^{(a,a-3)}\!\otimes\! D^{(2)}\!\otimes\! D^{(2)})\!\oplus\!(D^{(a-1,a-2)}\!\otimes\! D^{(1^2)}\!\otimes\! D^{(2)})\\
&\oplus\!(D^{(a,a-3)}\!\otimes\! D^{(2)}\!\otimes\! D^{(1^2)})\!\oplus\!(D^{(a-1,a-2)}\!\otimes\! D^{(2)}\!\otimes\! D^{(1^2)})\\
&\oplus\!(D^{(a-1,a-2)}\!\otimes\! D^{(1^2)}\!\otimes\! D^{(1^2)}),\\
D^{(a+2,a)}\da_{\s_{2a-2,2^2}}\!\cong & (D^{(a,a-2)}\!\otimes\! D^{(2)}\!\otimes\! D^{(2)})^2\!\oplus\!(D^{(a-1,a-1)}\!\otimes\! D^{(1^2)}\!\otimes\! D^{(2)})\\
&\oplus\!(D^{(a-1,a-1)}\!\otimes\! D^{(2)}\!\otimes\! D^{(2)})\!\oplus\!(D^{(a,a-2)}\!\otimes\! D^{(1^2)}\!\otimes\! D^{(2)})\\
&\oplus\!(D^{(a-1,a-1)}\!\otimes\! D^{(2)}\!\otimes\! D^{(1^2)})\!\oplus\!(D^{(a,a-2)}\!\otimes\! D^{(2)}\!\otimes\! D^{(1^2)})\\
&\oplus\!(D^{(a,a-2)}\!\otimes\! D^{(1^2)}\!\otimes\! D^{(1^2)}),\\
D^{(a+3,a)}\da_{\s_{2a-1,2^2}}\!\cong & (D^{(a+1,a-2)}\!\otimes\! D^{(2)}\!\otimes\! D^{(2)})\!\oplus\!(D^{(a,a-1)}\!\otimes\! D^{(2)}\!\otimes\! D^{(2)})\\
&\oplus\!(D^{(a,a-1)}\!\otimes\! D^{(1^2)}\!\otimes\! D^{(2)})\!\oplus\!(D^{(a,a-1)}\!\otimes\! D^{(2)}\!\otimes\! D^{(1^2)})\\
&\oplus\!(D^{(a+1,a-2)}\!\otimes\! D^{(1^2)}\!\otimes\! D^{(1^2)}).
\end{align*}
The only possible composition factors of $D^\lambda\da_{\s_4}$ are $D^{(4)}$, $D^{(3,1)}$ and $D^{(2^2)}$. So since 
\begin{align*}
D^{(4)}\da_{\s_{2^2}}\!\cong & D^{(2)}\!\otimes \!D^{(2)},\\
D^{(3,1)}\da_{\s_{2^2}}\!\cong &(D^{(2)}\!\otimes\! D^{(2)})\!\oplus\!(D^{(2)}\!\otimes \!D^{(1^2)})\!\oplus\!(D^{(1^2)}\!\otimes \!D^{(2)}),\\
D^{(2^2)}\da_{\s_{2^2}}\!\cong &(D^{(2)}\!\otimes \!D^{(2)})\!\oplus\!(D^{(1^2)}\!\otimes \!D^{(1^2)}),
\end{align*}
the structure of $D^\lambda\da_{\s_{n-4,4}}$ follows.
\end{proof}

\begin{lemma}\label{l49}
Let $p=5$ and $n\equiv \pm 1\Md 5$ with $n\geq 9$. If $\lambda=(a+b,a)\vdash n$ with $0\leq b\leq 3$ then there exists $\psi\in\Hom_{\s_n}(M^{(n-3,1^3)},\End_F(D^\lambda))$ which does not vanish on $S^{(n-3,1^3)}$.
\end{lemma}

\begin{proof}
From Lemma \ref{l27} it can be checked that
\begin{align*}
\dim\End_{\s_{n-3}}(\hspace{-1pt}D^\lambda\da_{\s_{n-3}}\hspace{-1pt})\hspace{-1pt}&\!>\!2\dim\End_{\s_{n-3,2}}(\hspace{-1pt}D^\lambda\da_{\s_{n-3,2}}\hspace{-1pt})\!+\!\dim\End_{\s_{n-2}}(\hspace{-1pt}D^\lambda\da_{\s_{n-2}}\hspace{-1pt})\\
&\hspace{12pt}\!\!-\!\dim\End_{\s_{n-3,3}}(\hspace{-1pt}D^\lambda\da_{\s_{n-3,3}}\hspace{-1pt})\!-\!\dim\End_{\s_{n-2,2}}(\hspace{-1pt}D^\lambda\da_{\s_{n-2,2}}\hspace{-1pt})\\
&\hspace{12pt}\!\!-\!\dim\End_{\s_{n-1}}(\hspace{-1pt}D^\lambda\da_{\s_{n-1}}\hspace{-1pt})\!+\!1,
\end{align*}
and so the result holds by Lemma \ref{c2}.
\end{proof}

\begin{lemma}\label{l49'}
Let $p=5$ and $n\equiv 0\Md 5$ with $n\geq 9$. If $\lambda=(a+b,a)\vdash n$ with $0\leq b\leq 3$ then there exists $\psi\in\Hom_{\s_n}(M^{(n-4,2^2)},\End_F(D^\lambda))$ which does not vanish on $S^{(n-4,2^2)}$.
\end{lemma}

\begin{proof}
From Lemma \ref{l27} it can be checked that
\begin{align*}
\dim\End_{\s_{n-4,2^2}}\hspace{-1pt}(\hspace{-1pt}D^\lambda\da_{\s_{n-4,2^2}}\hspace{-1pt})\!>&\dim\End_{\s_{n-4,3}}\hspace{-1pt}(\hspace{-1pt}D^\lambda\da_{\s_{n-4,3}}\hspace{-1pt})\!\hspace{-0.5pt}+\!\dim\End_{\s_{n-3,2}}\hspace{-1pt}(\hspace{-1pt}D^\lambda\da_{\s_{n-3,2}}\hspace{-1pt})\!\hspace{-0.5pt}\\
&+\!\dim\End_{\s_{n-2,2}}\hspace{-1pt}(\hspace{-1pt}D^\lambda\da_{\s_{n-2,2}}\hspace{-1pt})\!\hspace{-0.5pt}-\!\dim\End_{\s_{n-3,3}}\hspace{-1pt}(\hspace{-1pt}D^\lambda\da_{\s_{n-3,3}}\hspace{-1pt})\!\hspace{-0.5pt}\\
&-\!\dim\End_{\s_{n-2}}\hspace{-1pt}(\hspace{-1pt}D^\lambda\da_{\s_{n-2}}\hspace{-1pt})
\end{align*}
and so the result holds by Lemma \ref{c1'}.
\end{proof}

\section{Mullineux fixed JS-partitions}\label{s4}

Mullineux fixed JS-partitions also play a special role in the proof of Theorem \ref{t3} and so will be studied in this section.

\begin{lemma}\label{l23}
Let $p\geq 3$ and $\lambda=\lambda^\Mull\vdash n$ be a JS-partition. Then $n\equiv h(\lambda)^2\Md p$.
\end{lemma}

\begin{proof}
Let $\lambda^0:=\lambda$ and then define recursively $\lambda^i$ to be obtained from $\lambda^{i-1}$ by removing the $p$-rim. From Theorem 4.1 of \cite{bo} we have that all the partitions $\lambda^i$ are Mullineux fixed JS-partitions. Further if $k$ is maximal such that $\lambda^k\not=()$, then $\lambda^k=(1)$. In particular $|\lambda^k|\equiv h(\lambda^k)^2\Md p$.

Assume that for a certain $1\leq i\leq k$ we have that $|\lambda^i|\equiv h(\lambda^i)^2\Md p$. From Theorem 4.1 of \cite{bo}, one of the following holds:
\begin{enumerate}
\item
$|\lambda^{i-1}|-|\lambda^i|\equiv 2h(\lambda^i)+1\Md p$ and $h(\lambda^{i-1})\equiv h(\lambda^i)+1\Md p$,

\item
$|\lambda^{i-1}|-|\lambda^i|\equiv -2h(\lambda^i)+1\Md p$ and $h(\lambda^{i-1})\equiv -h(\lambda^i)+1\Md p$,

\item
$h(\lambda^i)\equiv 0\Md p$, $|\lambda^{i-1}|-|\lambda^i|\equiv 0\Md p$ and $h(\lambda^{i-1})\equiv 0\Md p$.
\end{enumerate}
Using $|\lambda^i|\equiv h(\lambda^i)^2\Md p$ it follows that in each of the above cases:
\begin{enumerate}
\item
$|\lambda^{i-1}|\equiv |\lambda^i|+2h(\lambda^i)+1\equiv h(\lambda^i)^2+2h(\lambda^i)+1\equiv h(\lambda^{i-1})^2\Md p$,

\item
$|\lambda^{i-1}|\equiv |\lambda^i|-2h(\lambda^i)+1\equiv h(\lambda^i)^2-2h(\lambda^i)+1\equiv h(\lambda^{i-1})^2\Md p$,

\item
$|\lambda^{i-1}|\equiv|\lambda^i|\equiv 0\equiv h(\lambda^{i-1})\Md p$.
\end{enumerate}
The lemma then follows by induction.
\end{proof}

\begin{lemma}\label{l29}
Let $p=5$, $n\geq 5$ and $\lambda=\lambda^\Mull\vdash n$ be a JS-partition. Then there exists $i=\pm 1$ such that the following hold:
\begin{itemize}
\item
$D^\lambda\da_{\s_{n-1}}\cong D^{\tilde{e}_0(\lambda)}$,

\item
$\epsilon_{\pm i}(\tilde{e}_0(\lambda))=1$, $\epsilon_{j}(\tilde{e}_0(\lambda))=0$ for $j\not=\pm i$ and
\[D^\lambda\da_{\s_{n-2,2}}\cong (D^{\tilde{e}_i\tilde{e}_0(\lambda)}\otimes D^{(2)})\oplus (D^{\tilde{e}_{-i}\tilde{e}_0(\lambda)}\otimes D^{(1^2)}),\]

\item
$\epsilon_{-i}(\tilde{e}_i\tilde{e}_0(\lambda)),\epsilon_{2i}(\tilde{e}_i\tilde{e}_0(\lambda))=1$, $\epsilon_j(\tilde{e}_i\tilde{e}_0(\lambda))=0$ for $j\not=-i,2i$. Further $\tilde{e}_{-i}\tilde{e}_i\tilde{e}_0(\lambda)=\tilde{e}_i\tilde{e}_{-i}\tilde{e}_0(\lambda)$ and
\[D^\lambda\da_{\s_{n-3,3}}\cong (D^{\tilde{e}_{2i}\tilde{e}_i\tilde{e}_0(\lambda)}\otimes D^{(3)})\oplus (D^{\tilde{e}_{-i}\tilde{e}_i\tilde{e}_0(\lambda)}\otimes A)\oplus (D^{\tilde{e}_{-2i}\tilde{e}_{-i}\tilde{e}_0(\lambda)}\otimes D^{(1^3)}),\]
with $A\in\{D^{(2,1)},D^{(3)}\oplus D^{(1^3)}\}$.
\end{itemize}
\end{lemma}

\begin{proof}
Notice that from Lemma \ref{l17} the unique normal node of $\lambda$ has residue 0. So from Lemmas \ref{l45} and \ref{l39}, $D^\lambda\da_{\s_{n-1}}\cong D^{\tilde{e}_0(\lambda)}$. From Proposition 3.6 of \cite{ks2} we also have that
\[D^{\tilde{e}_0(\lambda)}\da_{\s_{n-2}}\cong D^{\lambda}\da_{\s_{n-2}}\cong D^\nu\oplus D^{\nu^\Mull}\]
with $\nu\not=\nu^\Mull$. From Lemmas \ref{l45} and \ref{l39} it then follows that there exist $i\not=k$ with $\epsilon_i(\tilde{e}_0(\lambda)),\epsilon_k(\tilde{e}_0(\lambda))=1$ and $\epsilon_j(\tilde{e}_0(\lambda))=0$ else. From Lemma \ref{l17} we have that $\tilde{e}_0(\lambda)=\tilde{e}_0(\lambda)^\Mull$ and then $k=-i\not=0$.

Let $i$ be the residue of the top removable node of $\tilde{e}_0(\lambda)$, which is always a normal node. We will prove that $i=\pm1$, that $\epsilon_{-i}(\tilde{e}_i\tilde{e}_0(\lambda)),\epsilon_{2i}(\tilde{e}_i\tilde{e}_0(\lambda))=1$ and that $\epsilon_j(\tilde{e}_i\tilde{e}_0(\lambda))=0$ else. Further we will prove that $\tilde{e}_{-i}\tilde{e}_i\tilde{e}_0(\lambda)=\tilde{e}_i\tilde{e}_{-i}\tilde{e}_0(\lambda)$. In view of Lemmas \ref{l45}, \ref{l39} and \ref{l17} it will follow that
\[D^\la\da_{\s_{n-3}}\cong D^{\tilde{e}_{2i}\tilde{e}_i\tilde{e}_0(\la)}\oplus (D^{\tilde{e}_{-i}\tilde{e}_i\tilde{e}_0(\la)})^2\oplus D^{\tilde{e}_{-2i}\tilde{e}_{-i}\tilde{e}_0(\la)}\]
and so the lemma will follow (up to exchange of $i$ and $-i$) by comparing $D^\lambda\da_{\s_{n-2,2}}\da_{\s_{n-3,1,2}}$ and $D^\lambda\da_{\s_{n-3,3}}\da_{\s_{n-3,1,2}}$.

Assume that $\epsilon_{-i}(\tilde{e}_i\tilde{e}_0(\lambda))=1$. Then $\epsilon_i(\tilde{e}_{-i}\tilde{e}_0(\lambda))=1$ by Lemma \ref{l17}. Let $A$ and $B$ be the normal node of $\tilde{e}_0(\lambda)$ of residue $i$ and $-i$ respectively. Then, from Lemma \ref{l12}, $A$ is normal in $\tilde{e}_{-i}\tilde{e}_0(\lambda)$ of residue $i$ and $B$ is normal in $\tilde{e}_i\tilde{e}_0(\lambda)$ of residue $-i$. Since $\epsilon_{\pm i}(\tilde{e}_0(\lambda)),\epsilon_{\mp i}(\tilde{e}_{\pm i}\tilde{e}_0(\lambda))=1$, it follows that
\[\tilde{e}_{-i}\tilde{e}_i\tilde{e}_0(\lambda)=\tilde{e}_{-i}(\tilde{e}_0(\lambda)\setminus A)=\tilde{e}_0(\lambda)\setminus\{A,B\}=\tilde{e}_i(\tilde{e}_0(\lambda)\setminus B)=\tilde{e}_i\tilde{e}_{-i}\tilde{e}_0(\lambda).\]
To prove the lemma it is then enough to prove that $i=\pm1$ and that $\epsilon_{-i}(\tilde{e}_i\tilde{e}_0(\lambda)),\epsilon_{2i}(\tilde{e}_i\tilde{e}_0(\lambda))=1$ and $\epsilon_j(\tilde{e}_i\tilde{e}_0(\lambda))=0$ else. From Lemma 1.8 of \cite{ks2} we have that $h(\lambda)\geq 4$ and then from Lemma 2.2 of \cite{bkz} that $\lambda_1\geq \lambda_4+2$, as otherwise $\lambda^\Mull_1=\lambda_1+\lambda_2+\lambda_3+\lambda_4>\lambda_1$, contradicting $\lambda=\lambda^\Mull$.

Write $\lambda=(a_1^{b_1},\ldots,a_h^{b_h})$ with $a_1>\ldots>a_h\geq 1$ and $1\leq b_j\leq 4$. From the previous part we have that $1\leq b_1\leq 3$ and that $h\geq 2$. Since $\lambda$ is a JS-partition we have from Theorem D of \cite{k2} we have that $b_1+b_2+a_1-a_2\equiv 0\Md 5$. If $a_1-a_2=1$ then we would have that $b_1+b_2=4$, and then $\lambda_1=a_1=a_2+1=\lambda_4$, leading to a contradiction. So $a_1\geq a_2+2$.  From Theorem D of \cite{k2} we also have that $(a_j^{b_j},\ldots,a_h^{b_h})$ is a JS-partition for each $1\leq j\leq h$. In particular if  $\nu=(\psi_1,\ldots,\psi_l,a_j^{b_j},\ldots,a_h^{b_h})$ is 5-regular with $\psi_l>a_j$ for some $1\leq j\leq h$ and some $l\geq 1$, then the only possible normal nodes of $\nu$ are the removable nodes in the first $l$ rows and the node $(l+b_j,a_j)$. This will be used in each of the following cases to find the normal nodes of $\tilde{e}_i\tilde{e}_0(\lambda)$.

Assume first that $b_1=3$. Then, since $\lambda$ is a JS-partition,
\[\begin{tikzpicture}
\draw (1.2,0) node {0};
\draw (1.2,0.4) node {1};
\draw (1.2,0.8) node {2};
\draw (1.6,0.8) node {3};
\draw (0.8,0) node {4};
\draw (-0.8,-1.2) node {3};
\draw (-0.3,-0.4) node {$\not=3$};
\draw (-2.8,-1.6) node {$\vdots$};
\draw (-1.8,-1.6) node {$\iddots$};
\draw (-3.3,-0.2) node {$\lambda=$};
\draw (2,-0.4) node {,};
\draw (-1.4,-1.4) -- (-0.6,-1.4) -- (-0.6,-0.2) -- (1.4,-0.2) -- (1.4,1) -- (-2.8,1) -- (-2.8,-1.4);
\end{tikzpicture}\hspace{12pt}\begin{tikzpicture}
\draw (1.2,0) node {0};
\draw (1.2,0.4) node {1};
\draw (1.2,0.8) node {2};
\draw (1.6,0.8) node {3};
\draw (0.8,0) node {4};
\draw (-0.8,-1.2) node {3};
\draw (-0.3,-0.4) node {$\not=3$};
\draw (-2.8,-1.6) node {$\vdots$};
\draw (-1.8,-1.6) node {$\iddots$};
\draw (-3.6,-0.2) node {$\tilde{e}_0(\lambda)=$};
\draw (2,-0.4) node {.};
\draw (-1.4,-1.4) -- (-0.6,-1.4) -- (-0.6,-0.2) -- (1,-0.2) -- (1,0.2) -- (1.4,0.2) -- (1.4,1) -- (-2.8,1) -- (-2.8,-1.4);
\end{tikzpicture}\]
So in this case $i=1$ and
\[\begin{tikzpicture}
\draw (1.2,0) node {0};
\draw (1.2,0.4) node {1};
\draw (1.2,0.8) node {2};
\draw (1.6,0.8) node {3};
\draw (0.8,0) node {4};
\draw (-0.8,-1.2) node {3};
\draw (-0.3,-0.4) node {$\not=3$};
\draw (-2.8,-1.6) node {$\vdots$};
\draw (-1.8,-1.6) node {$\iddots$};
\draw (-3.8,-0.2) node {$\tilde{e}_i\tilde{e}_0(\lambda)=$};
\draw (2,-0.4) node {,};
\draw (-1.4,-1.4) -- (-0.6,-1.4) -- (-0.6,-0.2) -- (1,-0.2) -- (1,0.6) -- (1.4,0.6) -- (1.4,1) -- (-2.8,1) -- (-2.8,-1.4);
\end{tikzpicture}\]
In particular $\tilde{e}_i\tilde{e}_0(\lambda)=(a_1,(a_1-1)^2,a_2^{b_2},\ldots,a_h^{b_h})$ with $a_1-1>a_2$. Since $(1,a_1)$ and $(3,a_1-1)$ are normal in $\tilde{e}_i\tilde{e}_0(\lambda)$ of residue 2 and 4 respectively while $(3+b_2,a_2)$ is not normal, the lemma follows in this case.

Assume next that $b_1=2$. Then, since $\lambda$ is a JS-partition,
\[\begin{tikzpicture}
\draw (1.2,0) node {0};
\draw (1.2,0.4) node {1};
\draw (1.6,0.4) node {2};
\draw (0.8,0) node {4};
\draw (-0.8,-1.2) node {2};
\draw (-0.3,-0.4) node {$\not=2$};
\draw (-2.8,-1.6) node {$\vdots$};
\draw (-1.8,-1.6) node {$\iddots$};
\draw (-3.3,-0.4) node {$\lambda=$};
\draw (2,-0.6) node {,};
\draw (-1.4,-1.4) -- (-0.6,-1.4) -- (-0.6,-0.2) -- (1.4,-0.2) -- (1.4,0.6) -- (-2.8,0.6) -- (-2.8,-1.4);
\end{tikzpicture}\hspace{12pt}\begin{tikzpicture}
\draw (1.2,0) node {0};
\draw (1.2,0.4) node {1};
\draw (1.6,0.4) node {2};
\draw (0.8,0) node {4};
\draw (-0.8,-1.2) node {2};
\draw (-0.3,-0.4) node {$\not=2$};
\draw (-2.8,-1.6) node {$\vdots$};
\draw (-1.8,-1.6) node {$\iddots$};
\draw (-3.6,-0.4) node {$\tilde{e}_0(\lambda)=$};
\draw (2,-0.6) node {.};
\draw (-1.4,-1.4) -- (-0.6,-1.4) -- (-0.6,-0.2) -- (1,-0.2) -- (1,0.2) -- (1.4,0.2) -- (1.4,0.6) -- (-2.8,0.6) -- (-2.8,-1.4);
\end{tikzpicture}\]
So in this case $i=1$ and
\[\begin{tikzpicture}
\draw (1.2,0) node {0};
\draw (1.2,0.4) node {1};
\draw (1.6,0.4) node {2};
\draw (0.8,0) node {4};
\draw (-0.8,-1.2) node {2};
\draw (-0.3,-0.4) node {$\not=2$};
\draw (-2.8,-1.6) node {$\vdots$};
\draw (-1.8,-1.6) node {$\iddots$};
\draw (-3.8,-0.4) node {$\tilde{e}_i\tilde{e}_0(\lambda)=$};
\draw (2,-0.6) node {,};
\draw (-1.4,-1.4) -- (-0.6,-1.4) -- (-0.6,-0.2) -- (1,-0.2) -- (1,0.6) -- (-2.8,0.6) -- (-2.8,-1.4);
\end{tikzpicture}\]
In particular $\tilde{e}_i\tilde{e}_0(\lambda)=((a_1-1)^2,a_2^{b_2},\ldots,a_h^{b_h})$ with $a_1-1>a_2$. Since $(2,a_1-1)$ and $(2+b_2,a_2)$ are normal in $\tilde{e}_i\tilde{e}_0(\lambda)$ of residue 4 and 2 respectively, the lemma follows in this case.

Assume now that $b_1=1$ and $a_1\geq a_2+3$. Then, being $\lambda$ a JS-partition,
\[\begin{tikzpicture}
\draw (2,0) node {1};
\draw (1.6,0) node {0};
\draw (1.2,0) node {4};
\draw (0.8,0) node {3};
\draw (-0.8,-1.2) node {1};
\draw (-0.3,-0.4) node {$\not=1$};
\draw (-2.8,-1.6) node {$\vdots$};
\draw (-1.8,-1.6) node {$\iddots$};
\draw (-3.3,-0.6) node {$\lambda=$};
\draw (2,-0.8) node {,};
\draw (-1.4,-1.4) -- (-0.6,-1.4) -- (-0.6,-0.2) -- (1.8,-0.2) -- (1.8,0.2) -- (-2.8,0.2) -- (-2.8,-1.4);
\end{tikzpicture}\hspace{12pt}\begin{tikzpicture}
\draw (2,0) node {1};
\draw (1.6,0) node {0};
\draw (1.2,0) node {4};
\draw (0.8,0) node {3};
\draw (-0.8,-1.2) node {1};
\draw (-0.8,-1.2) node {1};
\draw (-0.3,-0.4) node {$\not=1$};
\draw (-2.8,-1.6) node {$\vdots$};
\draw (-1.8,-1.6) node {$\iddots$};
\draw (-3.6,-0.6) node {$\tilde{e}_0(\lambda)=$};
\draw (2,-0.8) node {.};
\draw (-1.4,-1.4) -- (-0.6,-1.4) -- (-0.6,-0.2) -- (1.4,-0.2) -- (1.4,0.2) -- (-2.8,0.2) -- (-2.8,-1.4);
\end{tikzpicture}\]
So in this case $i=4$ and
\[\begin{tikzpicture}
\draw (2,0) node {1};
\draw (1.6,0) node {0};
\draw (1.2,0) node {4};
\draw (0.8,0) node {3};
\draw (-0.8,-1.2) node {1};
\draw (-0.8,-1.2) node {1};
\draw (-0.3,-0.4) node {$\not=1$};
\draw (-2.8,-1.6) node {$\vdots$};
\draw (-1.8,-1.6) node {$\iddots$};
\draw (-3.8,-0.6) node {$\tilde{e}_i\tilde{e}_0(\lambda)=$};
\draw (2,-0.8) node {,};
\draw (-1.4,-1.4) -- (-0.6,-1.4) -- (-0.6,-0.2) -- (1,-0.2) -- (1,0.2) -- (-2.8,0.2) -- (-2.8,-1.4);
\end{tikzpicture}\]
In particular $\tilde{e}_i\tilde{e}_0(\lambda)=(a_1-2,a_2^{b_2},\ldots,a_h^{b_h})$ with $a_1-2>a_2$. Since $(1,a_1-2)$ and $(1+b_2,a_2)$ are normal in $\tilde{e}_i\tilde{e}_0(\lambda)$ of residue 3 and 1 respectively, the lemma follows in this case.

Assume last that $b_1=1$ and $a_1=a_2+2$. Then $b_2=2$ (from $\lambda$ being a JS-partition) and $h\geq 3$ (as $\lambda$ has at least 4 rows). So
\[\begin{tikzpicture}
\draw (1.2,0) node {1};
\draw (0.8,0) node {0};
\draw (0.4,0) node {4};
\draw (0,0) node {3};
\draw (0.4,-0.4) node {3};
\draw (0,-0.4) node {2};
\draw (0,-0.8) node {1};
\draw (-1.6,-2) node {3};
\draw (-1.1,-1.2) node {$\not=3$};
\draw (-3.6,-2.4) node {$\vdots$};
\draw (-2.6,-2.4) node {$\iddots$};
\draw (-4.1,-1) node {$\lambda=$};
\draw (1.1,-1.2) node {,};
\draw (-2.2,-2.2) -- (-1.4,-2.2) --  (-1.4,-1) -- (0.2,-1) -- (0.2,-0.2) -- (1,-0.2) -- (1,0.2) -- (-3.6,0.2) -- (-3.6,-2.2);
\end{tikzpicture}\hspace{12pt}\begin{tikzpicture}
\draw (1.2,0) node {1};
\draw (0.8,0) node {0};
\draw (0.4,0) node {4};
\draw (0,0) node {3};
\draw (0.4,-0.4) node {3};
\draw (0,-0.4) node {2};
\draw (0,-0.8) node {1};
\draw (-1.6,-2) node {3};
\draw (-1.1,-1.2) node {$\not=3$};
\draw (-3.6,-2.4) node {$\vdots$};
\draw (-2.6,-2.4) node {$\iddots$};
\draw (-4.4,-1) node {$\tilde{e}_0(\lambda)=$};
\draw (1.1,-1.2) node {.};
\draw (-2.2,-2.2) -- (-1.4,-2.2) --  (-1.4,-1) -- (0.2,-1) -- (0.2,-0.2) -- (0.6,-0.2) -- (0.6,0.2) -- (-3.6,0.2) -- (-3.6,-2.2);
\end{tikzpicture}\]
In this case $i=4$ and
\[\begin{tikzpicture}
\draw (1.2,0) node {1};
\draw (0.8,0) node {0};
\draw (0.4,0) node {4};
\draw (0,0) node {3};
\draw (0.4,-0.4) node {3};
\draw (0,-0.4) node {2};
\draw (0,-0.8) node {1};
\draw (-1.6,-2) node {3};
\draw (-1.1,-1.2) node {$\not=3$};
\draw (-3.6,-2.4) node {$\vdots$};
\draw (-2.6,-2.4) node {$\iddots$};
\draw (-4.6,-1) node {$\tilde{e}_i\tilde{e}_0(\lambda)=$};
\draw (1.5,-1.2) node {.};
\draw (-2.2,-2.2) -- (-1.4,-2.2) --  (-1.4,-1) -- (0.2,-1) -- (0.2,0.2) -- (-3.6,0.2) -- (-3.6,-2.2);
\end{tikzpicture}\]
Since $\tilde{e}_i\tilde{e}_0(\lambda)=(a_2^3,a_3^{b_3},\ldots,a_h^{b_h})$ and the nodes $(3,a_2)$ and $(3+b_3,a_3)$ are normal in $\tilde{e}_i\tilde{e}_0(\lambda)$ of residue 1 and 3 respectively the lemma follows also in this case.
\end{proof}

\begin{lemma}\label{l50}
Let $p=5$ and $n\equiv \pm 1\Md 5$ with $n\geq 6$. If $\lambda=\lambda^\Mull\vdash n$ then there exists $\psi\in\Hom_{\s_n}(M^{(n-3,1^3)},\End_F(D^\lambda))$ which does not vanish on $S^{(n-3,1^3)}$.
\end{lemma}

\begin{proof}
From Lemma \ref{l29}
\begin{align*}
\dim\End_{\s_{n-3}}(\hspace{-1pt}D^\lambda\da_{\s_{n-3}}\hspace{-1pt})\hspace{-1pt}&\!>\!2\dim\End_{\s_{n-3,2}}(\hspace{-1pt}D^\lambda\da_{\s_{n-3,2}}\hspace{-1pt})\!+\!\dim\End_{\s_{n-2}}(\hspace{-1pt}D^\lambda\da_{\s_{n-2}}\hspace{-1pt})\\
&\hspace{12pt}\!\!-\!\dim\End_{\s_{n-3,3}}(\hspace{-1pt}D^\lambda\da_{\s_{n-3,3}}\hspace{-1pt})\!-\!\dim\End_{\s_{n-2,2}}(\hspace{-1pt}D^\lambda\da_{\s_{n-2,2}}\hspace{-1pt})\\
&\hspace{12pt}\!\!-\!\dim\End_{\s_{n-1}}(\hspace{-1pt}D^\lambda\da_{\s_{n-1}}\hspace{-1pt})\!+\!1,
\end{align*}
and so the result holds by Lemma \ref{c2}.
\end{proof}

\section{Split-non-split case}\label{s1}

In this section we will prove Theorem 1.1 in the case where one of the two irreducible $A_n$-modules $D_1,D_2$ splits when reduced to $A_n$, while the other does not.

\begin{lemma}\label{l14}
Let $p\geq 3$ and $\lambda,\mu\vdash n$ be $p$-regular. If $\lambda=\lambda^\Mull$, $\mu\not=\mu^\Mull$ and $E^\lambda_\pm\otimes E^\mu$ is irreducible then $D^\lambda\otimes D^\mu\cong D^\nu\oplus D^{\nu^\Mull}$ for some $\nu\not=\nu^\Mull$. In particular
\begin{align*}
\dim\Hom_{\s_n}(\End_F(D^\lambda),\End_F(D^\mu))&=2.
\end{align*}
\end{lemma}

\begin{proof}
See Lemma 3.1 of \cite{bk2}.
\end{proof}

This lemma will be used in Lemma \ref{t4} and Theorem \ref{tsns} to restrict the number of tensor products that have to be considered.

\begin{lemma}\label{t4}
Let $p\geq 3$, $n\geq 6$ and $\lambda,\mu\vdash n$ be $p$-regular. If $\lambda=\lambda^\Mull$, $\mu\not=\mu^\Mull$, $E^\lambda_\pm$ and $E^\mu$ are not 1-dimensional and $E^\lambda_\pm\otimes E^\mu$ is irreducible then $\lambda$ is a JS-partition.
\end{lemma}

\begin{proof}
From Lemmas \ref{l1l2} and \ref{l30} there then exists $\psi_{\mu}:M^{(n-2,2)}\rightarrow\End_F(D^\lambda)$ which does not vanish on $S^{(n-2,2)}$. If $\lambda$ is not a JS-partition, from Lemma \ref{l1l2} and one of Lemmas \ref{l20}, \ref{l21} and \ref{l22} there exist $\psi_{\lambda,1},\psi_{\lambda,2}:M^{(n-2,2)}\rightarrow\End_F(D^\lambda)$ which are linearly independent when restricted to $S^{(n-2,2)}$.

So from Lemma \ref{l15}, considering $\alpha=(n)$ and $(n-2,2)$,  it follows that
\[\dim\Hom_{\s_n}(\End_F(D^\lambda),\End_F(D^\mu))\geq 3\]
if $\lambda$ is not a JS-partition. The result then follows by Lemma \ref{l14}.
\end{proof}

\begin{theor}\label{tsns}
Let $p=5$. Let $\lambda,\mu$ be 5-regular partitions with $\lambda=\lambda^\Mull$ and $\mu\not=\mu^\Mull$ such that $E^\lambda_\pm$ and $E^\mu$ are not 1-dimensional. If $E^\lambda_\pm\otimes E^\mu$ is irreducible then $\lambda$ is a JS-partition and $\mu\in\{(n-1,1),(n-1,1)^\Mull\}$.
\end{theor}

\begin{proof}
For $n\leq 7$ the lemma follows by considering each case separately. So we can assume that $n\geq 8$. By Lemma \ref{t4} we have that $\lambda$ is a JS-partition. So $n\equiv 0,1\mbox{ or }4\Md 5$ by Lemma \ref{l23}. From Lemma 1.8 of \cite{ks2} we have that $h(\la)\geq 4$. Further from Lemmas \ref{l1l2} and \ref{l30} there exist $\psi_{\lambda,2}:M^{(n-2,2)}\to \End_F(D^\lambda)$ and $\psi_{\mu,2}:M^{(n-2,2)}\to \End_F(D^\mu)$ which do not vanish on $S^{(n-2,2)}$.

If $\mu,\mu^\Mull\not=(n-k,k)$ with $k=1$ or $n-2k\leq 3$ then, from Corollaries 3.9 and 4.12 of \cite{bk5} and Lemma \ref{l1}, for some $j\in\{3,4\}$ there exist $\psi_{\lambda,j}:M^{(n-j,j)}\to \End_F(D^\lambda)$ and $\psi_{\mu,j}:M^{(n-j,j)}\to \End_F(D^\mu)$ which do not vanish on $S^{(n-j,j)}$.

If $\mu,\mu^\Mull=(n-k,k)$ with $n-2k\leq 3$ and $n\equiv 0\Md 5$ then there exist $\psi_{\lambda,2^2}:M^{(n-4,2^2)}\to \End_F(D^\lambda)$ and $\psi_{\mu,2^2}:M^{(n-4,2^2)}\to \End_F(D^\mu)$ which do not vanish on $S^{(n-4,2^2)}$ by Lemmas \ref{l51} and \ref{l49'}.

If $\mu,\mu^\Mull=(n-k,k)$ with $n-2k\leq 3$ and $n\equiv \pm 1\Md 5$ then there exist $\psi_{\lambda,1^3}:M^{(n-3,1^3)}\to \End_F(D^\lambda)$ and $\psi_{\mu,1^3}:M^{(n-3,1^3)}\to \End_F(D^\mu)$ which do not vanish on $S^{(n-3,1^3)}$ by Lemmas \ref{l49} and \ref{l50}.

In either of these cases it follows from Lemma \ref{l15}, also considering $\alpha=(n)$, that
\[\dim\Hom_{\s_n}(\End_F(D^\lambda),\End_F(D^\mu))\geq 3.\]
So from Lemma \ref{l14} we have that if $E^\lambda\pm\otimes E^\mu$ is irreducible then $\mu\in\{(n-1,1),(n-1,1)^\Mull\}$. 
\end{proof}

\begin{theor}\label{tsnat}
Let $p\geq 3$ and $\lambda$ be a $p$-regular partitions with $\lambda=\lambda^\Mull$. Then $E^\lambda_\pm\otimes E^{(n-1,1)}$ is irreducible if and only if $n\not\equiv 0\Md p$ and $\lambda$ is a JS-partition. In this case, if $\nu$ is obtained from $\lambda$ by removing the top removable node and adding the bottom addable node, then $E^\lambda_\pm\otimes E^{(n-1,1)}\cong E^\nu$.
\end{theor}

\begin{proof}
See Theorem 3.3 of \cite{bk2} and Lemma \ref{l23}.
%
%
\end{proof}

\section{Double-split case}\label{s2}

In this section we will prove Theorem 1.1 in the case where both irreducible $A_n$-modules $D_1,D_2$ split when reduced to $A_n$.

\begin{lemma}\label{l7}
Let $\lambda,\mu$ be $p$-regular partitions with $\lambda=\lambda^\Mull$ and $\mu=\mu^\Mull$. Also let $\epsilon_1,\epsilon_2\in\{\pm\}$. If $E^\lambda_{\epsilon_1}\otimes E^\mu_{\epsilon_2}$ is irreducible then
\[\dim\Hom_{A_n}(\Hom_F(E^\lambda_{\epsilon_1},E^\lambda_{\delta_1}),\Hom_F(E^\mu_{\delta_2},E^\mu_{\epsilon_2}))\leq 1\]
for any combination $\delta_1,\delta_2\in\{\pm\}$. In particular
\[\dim\Hom_{A_n}(\Hom_F(E^\lambda_{\epsilon_1},E^\lambda_+\oplus E^\la_-),\Hom_F(E^\mu_+\oplus E^\mu_-,E^\mu_{\epsilon_2}))\leq 4.\]
\end{lemma}

\begin{proof}
See Lemma 3.4 of \cite{bk2} (and its proof).
\end{proof}

\begin{lemma}\label{l19}
Let $p\geq 3$ and $n\geq 6$. Let $\lambda,\mu$ be $p$-regular partitions with $\lambda=\lambda^\Mull$ and $\mu=\mu^\Mull$. Assume that $E^\lambda_{\epsilon_1}\otimes E^\mu_{\epsilon_2}$ is irreducible for some $\epsilon_1,\epsilon_2\in\{\pm\}$. Then, up to exchange of $\lambda$ and $\mu$,
\begin{align*}
\dim\End_{\s_{n-2,2}}(D^\lambda\da_{\s_{n-2,2}})&=\dim\End_{\s_{n-1}}(D^\lambda\da_{\s_{n-1}})+1,\\
\dim\End_{\s_{n-2,2}}(D^\mu\da_{\s_{n-2,2}})&\leq\dim\End_{\s_{n-1}}(D^\mu\da_{\s_{n-1}})+2.
\end{align*}
\end{lemma}

\begin{proof}
Notice first that $(n)>(n)^\Mull$ and $(n-2,2)>(n-2,2)^\Mull$ (this follows from Lemma 1.8 of \cite{ks2}, since if $n=a(p-1)+b$ with $0\leq b<p-1$ then $(n)^\Mull=((a+1)^b,a^{p-1-b})$, so that $(n-2,2)^\Mull\not=(n)$ as $n\geq 6$).

For any $\alpha\vdash n$, from Lemmas \ref{l2} and \ref{l6} we have that
\begin{align*}
\dim\End_{\s_\alpha}(D^\lambda\da_{\s_\alpha})&=\dim\Hom_{\s_n}(M^\alpha,\End_F(D^\lambda))\\
&=\dim\Hom_{A_n}(M^\alpha,\Hom_F(E^\lambda_+\oplus E^\lambda_-,E^\lambda_{\epsilon_1}))
\end{align*}
and similarly for $\mu$.

From Lemma \ref{l30} we have that
\begin{align*}
\dim\End_{\s_{n-2,2}}(D^\lambda\da_{\s_{n-2,2}})&\geq\dim\End_{\s_{n-1}}(D^\lambda\da_{\s_{n-1}})+1,\\
\dim\End_{\s_{n-2,2}}(D^\mu\da_{\s_{n-2,2}})&\geq\dim\End_{\s_{n-1}}(D^\mu\da_{\s_{n-1}})+1.
\end{align*}

Assume first that
\begin{align*}
\dim\End_{\s_{n-2,2}}(D^\lambda\da_{\s_{n-2,2}})&\geq\dim\End_{\s_{n-1}}(D^\lambda\da_{\s_{n-1}})+2,\\
\dim\End_{\s_{n-2,2}}(D^\mu\da_{\s_{n-2,2}})&\geq\dim\End_{\s_{n-1}}(D^\mu\da_{\s_{n-1}})+2.
\end{align*}
Then, from Lemmas \ref{l1} and \ref{l15} we have that
\[\dim\Hom_{A_n}(\Hom_F(E^\lambda_{\epsilon_1},E^\lambda_+\oplus E^\lambda_-),\Hom_F(E^\mu_+\oplus E^\mu_-,E^\lambda_{\epsilon_2}))\geq 1+2\cdot 2=5.\]
This contradicts $E^\lambda_{\epsilon_1}\otimes E^\mu_{\epsilon_2}$ being irreducible, due to Lemma \ref{l7}.

Up to exchange we can then assume that
\begin{align*}
\dim\End_{\s_{n-2,2}}(D^\lambda\da_{\s_{n-2,2}})&=\dim\End_{\s_{n-1}}(D^\lambda\da_{\s_{n-1}})+1,\\
\dim\End_{\s_{n-2,2}}(D^\mu\da_{\s_{n-2,2}})&\geq\dim\End_{\s_{n-1}}(D^\mu\da_{\s_{n-1}})+3.
\end{align*}
Then, from Lemma \ref{l2} and by self-duality of $M^{(n-1,1)}$ and $M^{(n-2,2)}$,
\begin{align*}
&\dim\Hom_{A_n}(\Hom_F(E^\lambda_{\epsilon_1},E^\lambda_+\oplus E^\lambda_-),M^{(n-2,2)})\\
&=\dim\Hom_{A_n}(\Hom_F(E^\lambda_{\epsilon_1},E^\lambda_+\oplus E^\lambda_-),M^{(n-1,1)})+1
\end{align*}
and
\begin{align*}
&\dim\Hom_{A_n}(M^{(n-2,2)},\Hom_F(E^\mu_+\oplus E^\mu_-,E^\mu_{\epsilon_2}))\\
&\geq\dim\Hom_{A_n}(M^{(n-1,1)},\Hom_F(E^\mu_+\oplus E^\mu_-,E^\lambda_{\epsilon_2}))+3.
\end{align*}
In particular there exist $\delta_1,\delta_2\in\{\pm\}$ with
\begin{align*}
&\dim\Hom_{A_n}(\Hom_F(E^\lambda_{\delta_1},E^\lambda_{\epsilon_1}),M^{(n-2,2)})\\
&\geq\dim\Hom_{A_n}(\Hom_F(E^\lambda_{\delta_1},E^\lambda_{\epsilon_1}),M^{(n-1,1)})+1
\end{align*}
and
\begin{align*}
&\dim\Hom_{A_n}(M^{(n-2,2)},\Hom_F(E^\mu_{\delta_2},E^\mu_{\epsilon_2}))\\
&\geq\dim\Hom_{A_n}(M^{(n-1,1)},\Hom_F(E^\mu_{\delta_2},E^\lambda_{\epsilon_2}))+2.
\end{align*}
From Lemmas \ref{l1} and \ref{l15} it then follows that
\[\dim\Hom_{A_n}(\Hom_F(E^\lambda_{\epsilon_1},E^\lambda_{\delta_1}),\Hom_F(E^\mu_{\delta_2},E^\lambda_{\epsilon_2}))\geq 2,\]
again contradicting that $E^\lambda_{\epsilon_1}\otimes E^\mu_{\epsilon_2}$ is irreducible, due to Lemma \ref{l7}.
\end{proof}

\begin{theor}\label{tds}
Let $p=5$. If $\lambda,\mu\vdash n$ are 5-regular partitions with $\lambda=\lambda^\Mull$ and $\mu=\mu^\Mull$ then $E^\lambda_{\epsilon_1}\otimes E^\mu_{\epsilon_2}$ is not irreducible for any choice of $\epsilon_1,\epsilon_2\in\{\pm\}$ unless $n\leq 4$, in which case $E^\lambda_\pm$ and $E^\mu_\pm$ are 1-dimensional.
\end{theor}

\begin{proof}
For $n\leq 11$ the lemma can be proved by considering each case separately. So assume now that $n\geq 12$. If $0\leq i\leq 4$ and $n-i=4a_i+b_i$ with $0\leq b_i\leq 3$, then $(n-i,i)^\Mull=((a_i+1)^{b_i},a_i^{4-b_i},1^i)$ by Lemma 2.3 of \cite{bkz}. In particular $(n-i,i)>(n-i,i)^\Mull$ for $0\leq i\leq 4$.

From Lemma 1.8 of \cite{ks2} we have that $h(\lambda),h(\mu)\geq 4$. So, from Corollary 3.9 of \cite{bk5},
\begin{align*}
\dim\Hom_{\s_n}(M^{(n-3,3)},\End_F(D^\lambda))&>\dim\Hom_{\s_n}(M^{(n-2,2)},\End_F(D^\lambda)),\\
\dim\Hom_{\s_n}(M^{(n-4,4)},\End_F(D^\lambda))&>\dim\Hom_{\s_n}(M^{(n-3,3)},\End_F(D^\lambda)),\\
\dim\Hom_{\s_n}(M^{(n-3,3)},\End_F(D^\mu))&>\dim\Hom_{\s_n}(M^{(n-2,2)},\End_F(D^\mu)),\\
\dim\Hom_{\s_n}(M^{(n-4,4)},\End_F(D^\mu))&>\dim\Hom_{\s_n}(M^{(n-3,3)},\End_F(D^\mu)).
\end{align*}

From Lemma \ref{l19} we can assume that
\[\dim\End_{\s_{n-2,2}}(D^\lambda\da_{\s_{n-2,2}})=\dim\End_{\s_{n-1}}(D^\lambda\da_{\s_{n-1}})+1.\]

Assume first that
\[\dim\End_{\s_{n-2,2}}(D^\mu\da_{\s_{n-2,2}})\geq\dim\End_{\s_{n-1}}(D^\mu\da_{\s_{n-1}})+2.\]
For any $\alpha\vdash n$ from Lemmas \ref{l2} and \ref{l6} we have that
\begin{align*}
\dim\End_{\s_\alpha}(D^\lambda\da_{\s_\alpha})&=\dim\Hom_{A_n}(M^\alpha,\Hom_F(E^\lambda_+\oplus E^\lambda_-,E^\lambda_{\epsilon_1}))\\
\dim\End_{\s_\alpha}(D^\mu\da_{\s_\alpha})&=\dim\Hom_{A_n}(M^\alpha,\Hom_F(E^\mu_+\oplus E^\mu_-,E^\mu_{\epsilon_1})).
\end{align*}
Since $(n-i,i)>(n-i,i)^\Mull$ by Lemmas \ref{l1} and \ref{l15} we have that
\[\dim\Hom_{\s_n}(\Hom_F(E^\lambda_{\epsilon_1},E^\lambda_+\oplus E^\lambda_-),\Hom_F(E^\mu_+\oplus E^\mu_-,E^\mu_{\epsilon_2}))\geq 1+0+2+1+1=5.\]
In particular, from Lemma \ref{l7}, $E^\lambda_\pm\otimes E^\mu_\pm$ is not irreducible.

So we may now assume (from Lemma \ref{l30}) that
\begin{align*}
\dim\End_{\s_{n-2,2}}(D^\lambda\da_{\s_{n-2,2}})&=\dim\End_{\s_{n-1}}(D^\lambda\da_{\s_{n-1}})+1,\\
\dim\End_{\s_{n-2,2}}(D^\mu\da_{\s_{n-2,2}})&=\dim\End_{\s_{n-1}}(D^\mu\da_{\s_{n-1}})+1.
\end{align*}
From Lemmas \ref{l20}, \ref{l21} and \ref{l22} we then have that $\lambda$ and $\mu$ are JS-partitions.

From Lemma \ref{l29} we have that $(E^\lambda_+\oplus E^\lambda_-)\da_{A_{n-2,2}}\cong D^\lambda\da_{A_{n-2,2}}$ has only 2 composition factors (since so does $D^\lambda\da_{\s_{n-2,2}}$ and none of these composition factors is fixed under tensoring with sign). In particular $E^\lambda_{\epsilon_1}\da_{A_{n-2,2}}$ is simple. From Lemma 1.1 of \cite{bk2} and from Lemma \ref{l29} we have that $(E^\lambda_+\oplus E^\lambda_-)\da_{A_{n-3,3}}\cong D^\lambda\da_{A_{n-3,3}}$ is semisimple and has at least 3 composition factors. In particular $E^\lambda_{\epsilon_1}\da_{A_{n-3,3}}$ is semisimple with at least 2 composition factors. So
\[\dim\End_{A_{n-3,3}}(E^\lambda_{\epsilon_1}\da_{A_{n-3,3}})>\dim\End_{A_{n-2,2}}(E^\lambda_{\epsilon_1}\da_{A_{n-2,2}}).\]
Similarly
\[\dim\End_{A_{n-3,3}}(E^\mu_{\epsilon_2}\da_{A_{n-3,3}})>\dim\End_{A_{n-2,2}}(E^\mu_{\epsilon_2}\da_{A_{n-2,2}}).\]

From Lemmas \ref{lM}, \ref{l1} and \ref{l15} it then follows that
\[\dim\End_{A_n}(E^\lambda_{\epsilon_1}\otimes E^\mu_{\epsilon_2})=\dim\Hom_{A_n}(\End_F(E^\lambda_{\epsilon_1}),\End_F(E^\mu_{\epsilon_2}))\geq 1+0+1=2\]
and so also in this case $E^\lambda_{\epsilon_1}\otimes E^\mu_{\epsilon_2}$ is not irreducible.
\end{proof}

\section*{Acknowledgements}

The author thanks Alexander Kleshchev for some comments on parts of the paper.

While finishing writing this paper the author was supported by the DFG grant MO 3377/1-1. The author was also supported by the NSF grant DMS-1440140 and Simons Foundation while in residence at the MSRI during the Spring 2018 semester.

\end{document}